\documentclass[11pt]{amsart}
\pdfoutput=1
\usepackage{amssymb,amsmath,amsthm,mathrsfs,xcolor}
\definecolor{links}{rgb}{0.5,0.2,0}
\definecolor{green}{rgb}{0.2,0.5,0.2}
\usepackage[colorlinks=true,allcolors=links]{hyperref}
\usepackage[alphabetic,abbrev]{amsrefs}
\usepackage{tikz,pgfplots}

\newcommand{\avg}[1]{\left\langle #1 \right \rangle}
\newcommand{\abs}[1]{\left\vert #1 \right \vert}
\newcommand{\norm}[1]{\left\Vert #1 \right \Vert}
\newcommand{\brk}[1]{\left[ #1 \right]}

\DeclareMathOperator{\BMO}{{BMO}}
\DeclareMathOperator{\BMOA}{{BMOA}}
\DeclareMathOperator{\VMOA}{{VMOA}}
\newcommand{\B}{\mathrm{B}}
\newcommand{\A}{\mathrm{A}}
\newcommand{\D}{\mathrm{D}}
\newcommand{\RH}{\mathrm{RH}}

\newcommand{\BP}{\Pi}
\newcounter{theorems}

\newcounter{other}
\numberwithin{other}{section}
\newtheorem{proposition}[other]{Proposition}
\newtheorem{theorem}[theorems]{Theorem}
\newtheorem*{theorem*}{Theorem}
\newtheorem*{proposition*}{Proposition}

\newtheorem*{corollary*}{Corollary}
\numberwithin{cor}{theorems}
\newtheorem{lemma}[other]{Lemma}
\theoremstyle{definition}
\newtheorem{remark}[other]{Remark}

\title[Weighted estimates on planar domains]{Weighted estimates for the Bergman projection on planar domains}
\author[A. W. Green]{A. Walton Green}
\thanks{A. W. Green's research partially supported by NSF grant DMS-2202813. }
\address{A. Walton Green \hfill\break\indent 
Department of Mathematics\hfill\break\indent 
Washington University in Saint Louis\hfill\break\indent 
1 Brookings Drive\hfill\break\indent 
Saint Louis, MO 63130, USA}
\email{\href{mailto:awgreen@wustl.edu}{\textnormal{awgreen@wustl.edu}}}
\author[N. A. Wagner]{Nathan A. Wagner}
\thanks{N. A. Wagner's research partially supported by NSF grant DMS-2203272. }
\address{ Nathan A. Wagner \hfill\break\indent 
Department of Mathematics \hfill\break\indent 
 Brown University \hfill\break\indent 
 151 Thayer Street\hfill\break\indent 
 Providence, RI, 02912 USA}
\email{\href{mailto:nathan_wagner@brown.edu}{\textnormal{nathan\_wagner@brown.edu}}}


\numberwithin{equation}{section}

\begin{document}

\begin{abstract}
We investigate weighted Lebesgue space estimates for the Bergman projection on a simply connected planar domain via the domain's Riemann map. We extend the bounds which follow from a standard change-of-variable argument in two ways. First, we provide a regularity condition on the Riemann map, which turns out to be necessary in the case of uniform domains, in order to obtain the full range of weighted estimates for the Bergman projection for weights in a B\'{e}koll\`{e}-Bonami-type class. Second, by slightly strengthening our condition on the Riemann map, we obtain the weighted weak-type (1,1) estimate as well. Our proofs draw on techniques from both conformal mapping and dyadic harmonic analysis.
\end{abstract}

\maketitle

\section{Introduction}
Let $\Omega$ be a bounded, simply connected domain in $\mathbb C$. A fundamental object of study in complex analysis is the Bergman projection $\BP_{\Omega}$, which is the orthogonal projection from $L^2(\Omega)$  (the Lebesgue space with respect to the Lebesgue area measure $dA$ restricted to $\Omega$)   to the closed subspace of holomorphic functions. Unweighted and weighted $L^p$ estimates for $\BP_{\Omega}$ have close connections to duality of Bergman spaces,   theory of conformal mappings , and complex partial differential equations. They have been investigated by many authors, particularly in the case $\Omega= \mathbb{D}.$ See references   \cites{bekolle1982, bekollebonami1978, borichev04, hedenmalm, lanzani-stein-2004, PR} for some fundamental results   pertaining to the regularity of the Bergman projection.

For each interval $I \subset \mathbb T$, with center $e^{\mathrm{i} \theta_0}$, the associated Carleson box is 
	\begin{equation}\label{e:QI} Q_I=\left\{re^{\mathrm{i} \theta} \in \mathbb{D}: 1-r \leq \ell(I), |\theta-\theta_0| \leq \frac{\ell(I)}{2} \right\}.\end{equation}
It is well-known that Carleson boxes form a Muckenhoupt basis, and moreover the corresponding Muckenhoupt weight condition for $1<p<\infty$ (denoted by $\sigma \in \B_p(\mathbb D)$) characterizes the continuity of $\BP_{\mathbb{D}}$ on $L^p(\mathbb{D},\sigma)$ \cite{bekollebonami1978,bekolle1982}. More precisely, for $1\le p<\infty$, we say a weight $\sigma$, which is a locally integrable, almost everywhere positive function on $\mathbb{D}$, belongs to the B\'ekoll\`e-Bonami class $\B_p(\mathbb{D})$ if the $\B_p(\mathbb{D})$ characteristic,
	\begin{equation}\label{e:BpD} \brk{\sigma}_{\B_p(\mathbb{D})} = \sup_{I} \avg{ \sigma }_{Q_I} \avg{\sigma^{-1}}_{\frac{1}{p-1},Q_I} \end{equation}
is finite, where we are using the following local average notation. For a measurable function $f$, $0 < p <\infty $, a weight $v$, and a measurable set $E$,
	\[ \avg{ f}_{p,v,E} = \frac{ \norm{ f }_{L^p(E,v)} }{\left( \int_E v \, dA \right) ^{1/p}} , \quad \norm{ f }_{L^p(E,v)} = \left( \int_E \abs{f}^p v \, dA \right)^{1/p}.\]
When $p=\infty$, we understand both quantities above to be the $L^\infty(E)=L^\infty(E,v)$ norm. The measure $dA$ is the area measure in $\mathbb C$ (normalized so that $A(\mathbb{D})=1$) and when $v$ is identically $1$ or $p=1$ we omit those parameters in the notation.
	
Quantitative weighted estimates for $\BP_{\mathbb D}$ in terms of the characteristic \eqref{e:BpD} are also well-understood. For $1<p<\infty$,
	\begin{equation}\label{e:w} [\sigma]_{\B_p(\mathbb{D})}^{\frac 1{2p}} \lesssim \norm{ \BP_{\mathbb D} }_{\mathcal L(L^p(\mathbb D,\sigma))} \lesssim [\sigma]_{\B_p(\mathbb{D})}^{\max\left\{1,\frac{1}{p-1}\right\}}. \end{equation}
The lower inequality can be derived from \cite{bekolle1982}*{Proposition 2} together with arguments in \cite{rtw2017}, and the upper was first established in \cite{PR} and was extended to the unit ball in $n$ dimensions in \cite{rtw2017} \footnote{The quantitative upper estimate is not simply a matter of tracking the constants in the qualitative case, but rather was delayed until the proof of the $\mathrm{A}_2$ conjecture in harmonic analysis \cite{hytonen12} and the ensuing ``sparse revolution'' that took place in the early 2010's \cites{conde-alonso2016pointwise,lerner-a2, lacey-a2,lerner2019intuitive}} . A weak-type analog of \eqref{e:w} also holds when $p=1$, see \cite{bekolle1982}, though the sharp dependence on the $\B_1(\mathbb D)$ characteristic is not known.   
	
The following question is natural: what is the appropriate generalization of \eqref{e:BpD} and \eqref{e:w} to general bounded simply connected planar domains $\Omega$? 
One immediate obstacle is the apparent lack of canonical replacement for Carleson boxes that respects the boundary geometry. 
A proposal was given by Burbea in \cite[Theorem 2]{burbea83} which we slightly simplify here.

Let $\psi$ be a conformal map from $\mathbb D$ onto $\Omega$. For $1\le p<\infty$, we say a weight $\sigma$ on $\Omega$ belongs to the B\'ekoll\`e-Bonami class $\B_p(\Omega)$ if the $\B_p(\Omega)$ characteristic,
	\begin{equation}\label{e:BpO} \brk{\sigma}_{\B_p(\Omega)} = \sup_{I} \avg{ \sigma }_{\psi(Q_I)} \avg{\sigma^{-1}}_{\frac{1}{p-1},\psi(Q_I)} \end{equation}
is finite. For the domains we consider, the images $\{\psi(Q_I)\}_{I}$ do indeed form a Muckenhoupt basis (in the sense of \cite[Definition 3.1]{cmp2011}) and furthermore $\brk{\sigma}_{\B_p(\Omega)}$ is, up to an absolute constant, independent of the choice of conformal map (see Proposition \ref{prop:muck} below).

By changing variables, $[\sigma]_{\B_p(\Omega)}$ can be related to a weighted characteristic on the unit disc. For weights $u,v$ on $\mathbb D$, and $1 \le p < \infty$, we say $u$ belongs to the weighted B\'ekoll\`e-Bonami class $\B_p(\mathbb D,v)$ if 
	\begin{equation}\label{e:Bpuv} [u]_{\B_p(\mathbb D,v)} = \sup_{I} \avg{ u }_{v,Q_I} \avg{u^{-1}}_{\frac{1}{p-1},v,Q_I} < \infty.\end{equation}
When $v$ is identically $1$, we remove it from the notation and recover \eqref{e:BpD}.	
A change of variable reveals
	\begin{equation}\label{e:intro-change} \avg{\sigma}_{p,\psi(Q_I)} = \avg{\sigma \circ \psi}_{p,|\psi'|^2,Q_I},\end{equation}
hence, for $1\leq p<\infty$, $u = (\sigma \circ \psi)$, and $v = |\psi'|^{2}$,
	\begin{equation}\label{e:sigma-uv} [\sigma]_{\B_p(\Omega)} = [u]_{\B_p(\mathbb D,v)}.\end{equation}
 
On the other hand, the Bergman projection $\BP_{\Omega}$ can be connected to $\BP_{\mathbb D}$ through the conformal map $\psi$. Consequently,
	\begin{equation}\label{e:transform} \norm{ \BP_{\Omega}}_{\mathcal L(L^p(\Omega,\sigma))} = \norm{ \BP_{\mathbb D} }_{\mathcal L(L^p(\mathbb D, (\sigma \circ \psi) |\psi|^{2-p}))}; \end{equation}
see, for example, \cite[Theorem 2]{burbea83} or \cite[Lemma 2.4]{wagner22}. 
So, in light of \eqref{e:sigma-uv}, \eqref{e:transform}, and \eqref{e:w}, we pose the following question. Under what conditions on a weight $v$ does it hold that $u \in \B_p(\mathbb D,v)$ implies $uv^{1-p/2} \in \B_p(\mathbb D)$?  
A   non-optimal   sufficient condition is that $v \in \B_1(\mathbb D)$. 
Indeed, the definition of $[v]_{\B_1(\mathbb D)}$ implies
\begin{equation}\label{e:introB1Bp} \avg{u v^{1-p/2}}_{Q_I} \avg{ uv^{1-p/2} }_{\frac{1}{p-1},Q_I}  \le [v]_{\B_1(\mathbb D)}^p [u]_{\B_p(\mathbb D,v)}, \qquad 1 \le p < \infty.\end{equation}
Thus, concerning the Bergman projection on $\Omega$, the following result   is an immediate consequence of   the above observations and the upper estimate in \eqref{e:w}.

\begin{theorem}\label{thm:strong-quant}
Let $\Omega \subset \mathbb C$ be a bounded, simply connected domain with $|\psi'|^2 \in \B_1(\mathbb D)$, then for $1<p<\infty$,
	\begin{equation}\label{e:strong-quant}\|\BP_{\Omega}\|_{\mathcal L(L^p(\Omega,\sigma))} \lesssim \left( \brk{\abs{\psi'}^2}_{\B_1(\mathbb D)}^p \brk{\sigma}_{\B_p(\Omega)} \right)^{\max\{1,\frac{1}{p-1}\}}.\end{equation}
Moreover, the implicit constant depends only on $p$.
\end{theorem}

In \cite{bekolle1986}, B\'ekoll\`e has shown that $|\psi'|^2 \in \B_1(\mathbb D)$ when $\Omega$ is convex. Also, when the boundary of $\Omega$ is sufficiently smooth (e.g. Dini smooth, which is slightly better than $C^1$), it is known that $\psi'$ extends to a continuous, non-vanishing function on $\overline{\mathbb{D}}$ \cite{pommerenke-book}, so trivially $|\psi'|^2 \in \B_1(\mathbb D)$. A weaker version of Theorem \ref{thm:strong-quant} in this latter special case of Dini smooth domains appears to have been discovered by Burbea in \cite{burbea80}. Beyond this, the authors are not aware of geometric conditions on $\Omega$ which are necessary or sufficient for $|\psi'|^2 \in \B_1(\mathbb D)$. 

Here, our goal is to go beyond the simple observations which to lead to Theorem \ref{thm:strong-quant} in two different, but related directions.

\subsection{Sharpened conditions on $\psi$ for weighted estimates}
First we will show that the assumption $|\psi'|^2 \in \B_1(\mathbb D)$ can be weakened to   an optimal condition (in some sense)   while maintaining the full range of weighted estimates. In fact, in Theorem \ref{thm:weighted-est} below, we give sufficient conditions on domains $\Omega$, in terms of $|\psi'|$, for which full or limited range weighted estimates hold for $\BP_{\Omega}$ with respect to the weight classes defined by \eqref{e:BpO}. We have the following special case of Theorem \ref{thm:weighted-est} for the full range of $p$:

\begin{theorem}\label{thm:mainsufffull} Let $\Omega \subset \mathbb{C}$ be a bounded, simply connected domain. If
	\begin{equation}\label{e:introB1plus-suff} |\psi'|^2 \in \bigcap_{p > 1} \B_p(\mathbb D),\end{equation}
then for each $\sigma \in \B_p(\Omega)$, $\BP_{\Omega} : L^p(\Omega,\sigma) \to L^p(\Omega,\sigma)$.
\end{theorem}

Theorem \ref{thm:mainsufffull} sharpens Theorem \ref{thm:strong-quant} in a non-trivial way. It is indeed the case that there exist domains $\Omega$ for which \eqref{e:introB1plus-suff} is satisfied, but $|\psi'|^2 \not\in \B_1(\mathbb D)$. See Section \ref{ss:example} below for an example. The main obstacle to obtaining the sufficiency in Theorem \ref{thm:mainsufffull} is that $\B_p(\mathbb D)$ weights do not in general satisfy a reverse H\"older inequality. If they did, then the strategy implemented by Johnson and Neugenbauer \cite{JN} concerning homeomorphisms which preserve the $\mathrm{A}_p(\mathbb R^n)$ classes could be easily repeated. In particular, a parallel theorem for the Szeg\H{o} projection can be established by directly adapting the arguments in the proof of \cite[Theorem 2.7]{JN}, so we do not focus on it in this paper. Nonetheless, we get around this difficulty for the Bergman projection by using the special form of the homeomorphism $\psi$, namely its conformality, as well as more modern tools from dyadic harmonic analysis. 

In the case when $\Omega$ is a uniform domain, the condition \eqref{e:introB1plus-suff} is sharp for weighted estimates (see Theorem \ref{thm:TFAE} below), and furthermore we can replace the $\B_p(\Omega)$ characteristic by the following equivalent one which is intrinsic to $\Omega$,
\begin{equation}\label{e:Dp} \brk{\sigma}_{\D_p(\Omega)} = \sup_{D} \avg{ \sigma }_{D \cap \Omega} \avg{\sigma^{-1}}_{\frac{1}{p-1}, D \cap \Omega} \end{equation}
 where the supremum is taken over all Euclidean disks $D$ centered on the boundary $\partial \Omega$. 

\begin{theorem}\label{thm:intro} Let $\Omega$ be a simply connected, bounded uniform domain. Then,
	\begin{equation}\label{e:introB1plus} |\psi'|^2 \in \bigcap_{p > 1} \B_p(\mathbb D),\end{equation}
if and only if for each $\sigma \in \D_p(\Omega)$, $\BP_{\Omega} : L^p(\Omega,\sigma) \to L^p(\Omega,\sigma)$.
\end{theorem}

We postpone the precise definition of uniform domain until Section \ref{s:prelim}, but the class includes all Lipschitz domains and furthermore, the boundary can be quite irregular, allowing for Hausdorff dimension arbitrarily close to $2$.
Furthermore, the condition \eqref{e:introB1plus} is satisfied whenever $\Omega$ is asymptotically conformal (see Section \ref{ss:asymp-conf} below) therefore the full range of weighted estimates hold on these domains.

 We remark in passing that our results should be extendable to unbounded graph domains. One only needs the existence of a conformal map that maps the upper-half space to $\Omega$, and has a continuous extension to the Riemann sphere that maps the real axis to $\partial \Omega \setminus \{\infty\}$ and fixes the point at infinity (such a conformal map is guaranteed to exist for graph domains, see \cite[Proposition 2.2]{lanzani-stein-2004}). In this case, one proceeds by using the upper-half plane rather than the unit disk as the model domain. The B\'{e}koll\`{e}-Bonami regions in the upper-half space are again Carleson tents, see \cite{PR}. We have not checked all the details, but we believe that the proofs should be completely analogous. 

\subsection{Weighted weak-type estimates}
Our second extension of Theorem \ref{thm:strong-quant} is to obtain the weighted weak-type $(1,1)$ analogue of \eqref{e:strong-quant}. The main difference is that when establishing the analogue of \eqref{e:transform} for the $L^{p,\infty}$ spaces, the conformal map appears partly as a multiplier, and partly as a measure (for the $L^p$ norm there is no distinction between a multiplier and a measure). The argument is similar to the one outlined in \cite[Lemma 3.1]{bekolle1986}, but we will need a generalization of the result there to weighted spaces and we include all the arguments for completeness. Weak-type estimates in which some or all of the weight is treated as a multiplier, which now appear to be called mixed weak-type inequalities, begin with Muckenhoupt and Wheeden \cite{muckenhoupt1977}. There was limited development of these types of estimates \cite{sawyer-weak,bonami82} until their systematic study by Cruz-Uribe, Martell, and P\'erez \cite{dcu-martell}. \cite{dcu-martell} and subsequent modern developments are motivated by interpolation problems with a change of measure, but as noted in Proposition \ref{prop:weak-transform} below, and in \cite{bonami82,bekolle1986}, they also arise in establishing weak-type estimates through a change of variable. The following theorem is proved in Section \ref{sec:weak} below. 

\begin{theorem}\label{thm:weak}
Let $\Omega \subset \mathbb C$ be a bounded, simply connected domain with $\abs{\psi'}^2 \in \B_1(\mathbb D)$, then
\begin{equation}\label{e:weak-quant}\|\BP_{\Omega}\|_{L^1(\Omega,\sigma) \to L^{1,\infty}(\Omega,\sigma)} \lesssim [\sigma]_{\B_1(\Omega)}^{3},\end{equation}
where the implicit constant depends polynomially on $\brk{\abs{\psi'}^2}_{\B_1(\mathbb D)}$.
\end{theorem}

\subsection{Organization}
This paper is organized as follows. In Section \ref{s:prelim}, we collect useful preliminaries regarding dyadic structures on $\mathbb{D}$, uniform domains, properties of conformal maps, and weight classes. In Section \ref{s:main-proof}, we introduce and prove a general version of Theorem \ref{thm:mainsufffull}, namely Theorem \ref{thm:weighted-est}. In Section \ref{sec:necessity}, we focus on the category of uniform domains and prove a converse to Theorem \ref{thm:weighted-est} which implies Theorem \ref{thm:intro}. Finally, in Section \ref{sec:weak} we prove Theorem \ref{thm:weak}.

\section{Preliminaries: Dyadic structures, conformal maps, and uniform domains}\label{s:prelim}
In this section we collect some preliminaries concerning dyadic harmonic analysis, conformal maps, and uniform domains.

\subsection{Dyadic harmonic analysis in $\mathbb D$} \label{ss:dyad}
Given a Carleson box $Q_I$, defined by \eqref{e:QI}, let $T_I$ denote the corresponding ``top half''
\[ T_I=\left \{re^{\mathrm{i} \theta} \in \mathbb{D}: \frac{\ell(I)}{2} \leq 1-r \leq \ell(I), |\theta-\theta_0| \leq \frac{\ell(I)}{2} \right\},\]
where $\ell(I)$ denotes the arclength of the interval $I$, normalized so that $\ell(\mathbb{T})=1$. 
$T_I$ coincides with $T_{I,\rho}$ from \cite{APR} with $\rho=\frac{1}{2}$. 

It is well-known that any Carleson box $Q_I$ can be well-approximated by $Q_J$ where $J$ belongs to one of two dyadic systems of intervals; say $\mathcal{D}_1, \mathcal{D}_2.$ By well-approximated, we mean that given an arbitrary interval $I$ on $\mathbb{T}$, there exists intervals $J,J' \in \mathcal{D}_1 \cup \mathcal{D}_2$ satisfying $J' \subset I \subset J$ and $|Q_I| \sim |Q_J| \sim |Q_{J'}|.$ For example, one can take 

\begin{equation}\label{e:dyad}\begin{gathered} \mathcal{D}_1= \left\{\left[\frac{2 \pi j}{2^k},\frac{2 \pi (j+1)}{2^k}\right): j,k \in \mathbb{N}, 0 \leq j< 2^k \right\},\\
\mathcal{D}_2= \left\{\left[\frac{2 \pi j}{2^k}+\frac{2 \pi}{3},\frac{2 \pi (j+1)}{2^k}+\frac{2 \pi}{3}\right): j,k \in \mathbb{N}, 0 \leq j< 2^k\right\}.\end{gathered} \end{equation}
 Therefore, when computing the B\'{e}koll\`{e}-Bonami characteristic of $u$, it suffices to compute averages over Carleson boxes $Q_I$ where $I$ belongs to $\mathcal{D}_1$ or $\mathcal{D}_2.$ The precise form of $\mathcal D_j$ is not important, but rather we note a few essential properties of any dyadic grid $\mathcal D$:
\begin{itemize}
	\item[(i)] For each $k \in \mathbb N$, $\{I \in \mathcal D : \ell(I) = 2^{-k} \}$ forms a partition of $\mathbb T$.
	\item[(ii)] $\{T_I : I \in \mathcal D\}$ forms a partition of $\mathbb D$.
	\item[(iii)] Any two $I,J \in \mathcal D$ are either disjoint, or one is contained in the other.
	\item[(iv)] For each $I \in \mathcal D$ and $k \ge 1$, we can subdivide $I$ into subintervals $\{I_j^k\}_{j=1}^{2^k}$ with side length $\ell(I_j^k)=2^{-k} \ell(I)$ which we call the \textit{$k$-th generation} of $I$. For $J=I_j^1$, we say $I$ is the \textit{dyadic parent} of $J$ and analogously, $Q_I$ the dyadic parent of $Q_J$.
\end{itemize}

Throughout $u$, $v$, and $w$ will be generic weights on $\mathbb D$.
We say a weight $u$ is doubling (with respect to $\mathcal D$) if there exists $c_u>0$ such that for all intervals $I\subset J$ in $\mathbb T$ (resp. in $\mathcal D$) with $\ell(J)=2\ell(I)$,
	\begin{equation}\label{e:doubling} c_u^{-1} \int_{Q_J} u \, dA  \le \int_{Q_I} u \, dA \le c_u \int_{T_I} u \, dA.\end{equation}
If only the first inequality in \eqref{e:doubling} holds, then we say $u$ is   \emph{weakly doubling}  . It is not difficult to show that any weight $u \in \B_p(\mathbb D)$ is doubling, $1\leq p<\infty.$

For a weight $v$ on $\mathbb{D}$, define the maximal operator 
\begin{equation}\label{e:max} M_{v} f(z):= \sup_{I: Q_I \ni z} \avg{f}_{v,Q_I}, \quad f\in L^1(\mathbb{D},\sigma), z \in \mathbb{D}.  \end{equation}
Note when $v \equiv 1$, $M_v$ corresponds to the ordinary maximal function in the Bergman setting, and in this case we simply write $M_v=M.$ When the supremum is restricted only to $I \in \mathcal D$, we use the notation $M^{\mathcal D}_v$. Standard considerations show that for any $v$ and any dyadic grid $\mathcal D$, $M^{\mathcal D}_v : L^p(\mathbb D,v) \to L^p(\mathbb D,v)$ with norm $p'=\frac{p}{p-1}$ \cite{lerner2019intuitive}*{Theorem 15.1}.
By the approximation property of $\mathcal{D}_1$ and $\mathcal{D}_2$,
\begin{equation} M_v f(z) \lesssim M_v^{\mathcal{D}_1}f(z)+ M_v^{\mathcal{D}_2}f(z), \quad f\in L^1(\mathbb{D},v), z \in \mathbb{D} \label{MaximalReduDyad} \end{equation}
holds whenever $v$ is doubling.

We also outline the following sparse estimate in the Bergman setting which we will use, but is quite standard by now. Let 
	\begin{equation}\label{e:A} \mathcal A_{\mathcal D,v}f = \sum_{I \in \mathcal D} \avg{ f}_{v,Q_I} \chi_{Q_I}.\end{equation}
Then,
	\[ \begin{aligned} \abs{ \left \langle \mathcal A_{\mathcal D,v} f,g \, wv \right\rangle } &= \sum_{I \in \mathcal D} \left( \int_{Q_I} v \, dA \right) \avg{ f }_{v,Q_I} \avg{ g w}_{v,Q_I} \\
		&\le c_v \sum_{I \in \mathcal D} \left(\int_{T_I} v \, dA \right) \avg{ w }_{v,Q_I} \avg{ w^{-1} }_{v,Q_I}\avg{ fw }_{w^{-1}v,Q_I} \avg{ g}_{wv,Q_I} \\
		&\le c_v \brk{w}_{\B_2(v)} \int_{\mathbb D} M_{w^{-1}v} (fw) w^{-\frac 12} M_{wv} (g) w^{\frac{1}{2}} v \, dA. \\
		&\le 4 c_v \brk{w}_{\B_2(v)} \norm {f}_{L^2(wv)} \norm{g}_{L^2(wv)}.
	\end{aligned} \]
This, together with a slight modification for $p \ne 2$ (see, e.g. \cite{moen2012}) gives
	\begin{equation}\label{e:Aw} \norm{ A_{\mathcal D,v} }_{L^p(\mathbb D,wv)} \le pp' c_v \brk{w}_{\B_p(v)}^{\max\left\{1,\frac{1}{p-1} \right\} }. \end{equation}

\subsection{Conformal estimates}\label{ss:conf}
 We first verify that the class $\B_p(\Omega)$ is independent of the choice of conformal map, and that the collection $\{\psi(Q_I)\}_{I \subseteq \mathbb{T}}$ forms a Muckenhoupt basis, which means the maximal function
	\begin{equation}\label{e:Mpsi} M_\psi f(z) = \sup_{I: z \in \psi(Q_I)} \avg{ f}_{\psi(Q_I)} \end{equation}
is bounded on $L^p(\Omega,\sigma)$ when $\sigma \in \B_p(\Omega)$. 
 
\begin{proposition}\label{prop:muck}
Suppose there exists a conformal map $\psi_0$ from $\mathbb D$ onto $\Omega$ such that $\abs{ \psi_0'}^2 \in \cup_{ q=1}^\infty \B_q(\mathbb D)$. Then for any other conformal map $\psi$ from $\mathbb D$ onto $\Omega$,
\begin{itemize} 
\item[(i)]For each $1<p<\infty$, there exists $C>0$ such that for any weight $\sigma$,
\[ \sup_{\psi} \sup_{I} \avg{ \sigma }_{\psi(Q_I)} \avg{\sigma^{-1} }_{\frac{1}{p-1}, \psi(Q_I)} \le C \sup_{I} \avg{ \sigma }_{\psi_0(Q_I)} \avg{\sigma^{-1} }_{\frac{1}{p-1},\psi_0(Q_I)}.\]
\item[(ii)]  $|\psi'|^2$ is doubling,
\item[(iii)] $\{ \psi(Q_I)\}_{I}$ forms a Muckenhoupt basis.
\end{itemize}
\end{proposition}

\begin{proof}
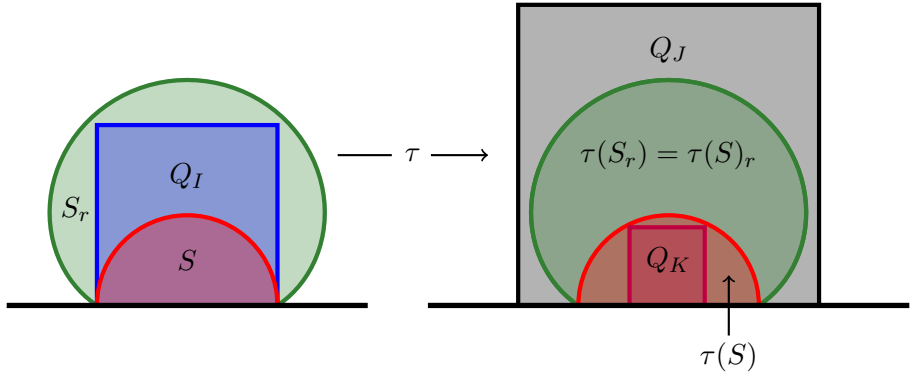
\begin{figure}\label{f:app}
\begin{tikzpicture}[scale=0.4]
\filldraw [green, ultra thick,  domain=0:180, samples=75, opacity=0.3] 
 plot ({
	3*(1+1.5*sin(\x))*cos(\x)
	},{
	3*(1+1.5*sin(\x))*sin(\x)
	});
\draw [green, ultra thick,  domain=0:180, samples=75, fill opacity=0.9] 
 plot ({
	3*(1+1.5*sin(\x))*cos(\x)
	},{
	3*(1+1.5*sin(\x))*sin(\x)
	});
\filldraw[blue,fill opacity=0.3] (-3,0) rectangle (3,6);
\draw[blue,fill opacity=0.9,ultra thick] (-3,0) rectangle (3,6);
\filldraw[red,fill opacity=0.3] (-3,0) -- (3,0) arc(0:180:3) --cycle;
\draw[red,fill opacity=0.9,ultra thick] (-3,0) -- (3,0) arc(0:180:3) --cycle;
\node at (0,1.5) {$S$};
\node at (0,4.3) {$Q_I$};
\node at (-3.7,3.2) {$S_r$};
\draw[line width=2pt] (-6,0) -- (6,0);
\draw[thick] (5,5) -- (6.8,5);
\node at (7.5,5) {$\tau$};
\draw[->,thick] (8.1,5) -- (10,5);
\draw[black,fill opacity=0.9,ultra thick] (11,0) rectangle (21,10);
\filldraw[black,fill opacity=0.3] (11,0) rectangle (21,10);
\filldraw [green, ultra thick,  domain=0:180, samples=75, fill opacity=0.3] 
 plot ({
	16+3*(1+1.5*sin(\x))*cos(\x)
	},{
	3*(1+1.5*sin(\x))*sin(\x)
	});
\draw [green, ultra thick,  domain=0:180, samples=75, fill opacity=0.9] 
 plot ({
	16+3*(1+1.5*sin(\x))*cos(\x)
	},{
	3*(1+1.5*sin(\x))*sin(\x)
	});
\filldraw[red,fill opacity=0.3] (13,0) -- (19,0) arc(0:180:3) --cycle;
\draw[red,fill opacity=0.9,ultra thick] (13,0) -- (19,0) arc(0:180:3) --cycle;
\filldraw[purple,fill opacity=0.3] (14.7,0) rectangle (17.2,2.6);
\draw[purple,fill opacity=0.9,ultra thick] (14.7,0) rectangle (17.2,2.6);

\node at (16,8.5) {$Q_J$};
\node at (16,5) {$\tau(S_r)=\tau(S)_r$};
\node at (16,1.5) {$Q_K$};
\node at (18,-1.7) {$\tau(S)$};
\draw[->,thick] (18,-1) to (18,1);
\draw[line width=2pt] (8,0) -- (24,0);
\end{tikzpicture}
\caption{Carleson boxes are well-approximated in the pseudohyperbolic metric by regions bounded by geodesics, $S$, in that $S \subset Q_I \subset S_r$ where $S_r$ is an $r$-neighborhood of $S$ in the pseudohyperbolic metric. Such $r>0$ is absolute and $|S| \sim |S_r|$. Since $\{S_r\}$ is invariant under $\tau$, the collection $\{Q_I\}$ is nearly invariant under automorphisms of $\mathbb D$ in the sense of \eqref{e:app}.}\end{figure}
 
First notice that the family $\{Q_I\}_I$ is nearly invariant under $\operatorname{Aut}(\mathbb D)$, the automorphisms of $\mathbb D$, in the sense that there exists an absolute constant $C>0$ such that given any $Q_I$ and any $\tau \in \operatorname{Aut}(\mathbb D)$, we can find $Q_J$ and $Q_K$ such that
	\begin{equation}\label{e:app} Q_{K} \subset \tau(Q_I) \subset Q_{J}, \qquad C^{-1} \le \frac{\ell( K)}{\ell(J)} \le C.\end{equation}
This geometric property is illustrated in Figure 1. Now, to prove (i), any such $\psi$ can be factored as $\psi=\psi_0 \circ \tau$ for some $\tau \in \operatorname{Aut}(\mathbb D)$, so by \eqref{e:app} and the doubling property of $\psi_0$,
	\[ \avg{\sigma}_{\psi(Q_I)} = \avg{\sigma}_{\psi_0( \tau (Q_I))} \le \frac{ \int_{\psi_0(Q_{J})} \sigma }{|\psi_0(Q_{K})|} \lesssim \avg{ \sigma }_{\psi_0(Q_{J})}.\]
Repeating the same estimate for $\sigma^{\frac{-1}{p-1}}$ establishes (i).

The proof of (ii) is immediate, since $\abs{\tau'}^2$ is in fact a $\B_1$ weight (with characteristic depending on $\tau$). Therefore $|\tau(Q_J)| \sim |\tau(Q_I)| \sim |\tau(T_I)|$.

Now, to prove (iii), let $f \in L^1(\Omega)$ be nonnegative,  $I \subset \mathbb{T}$, and $z \in \psi(Q_I).$ Then by change of variable,
\[ \avg{f}_{\psi(Q_I)}= \avg{f\circ \psi}_{v, Q_I} \leq M_{v}(f \circ \psi)(\psi^{-1}(z)), \qquad v = \abs{\psi'}^2,\]
which immediately implies the pointwise inequality 
\[ M_{\psi} f(z) 
\leq M_{v}(g)(\zeta), \quad g= f \circ \psi, \quad z \in \Omega. \]
Fix $\sigma \in \B_p(\Omega)$, which equivalently means $u = \sigma \circ \psi \in \B_p(\mathbb D,v)$ (recall \eqref{e:intro-change}). By \eqref{e:Aw} and the obvious inequality 
	\[ M_v g \lesssim \sum_{j=1,2} \sum_{I \in \mathcal{D}_j} \langle g \rangle_{v, Q_I} \chi_{Q_I},\] 
$M_v$ is bounded on $L^p(\mathbb{D}, uv)$. Then notice, again using change of variable, and assuming $f \in L^p(\Omega,\sigma)$, 

\begin{align*}
\|M_{\psi}f\|_{L^p(\Omega,\sigma)} & \leq \|[M_v g ] \circ \psi^{-1} \|_{L^p(\Omega,\sigma)} \\
& = \|M_v g\|_{L^p(\mathbb{D}, uv)}\\
& \lesssim \|g\|_{L^p(\mathbb{D}, uv)}\\
& = \|f\|_{L^p(\Omega, \sigma)}.
\end{align*}
\end{proof}

To conclude this section, we state a convenient version of the Koebe distortion theorem.

\begin{proposition}\label{KoebeDistort}
There exists an absolute constant $K$ so that for all $I \in \mathbb{T}$, all $z_1,z_2 \in T_I$, and any $\psi: \mathbb D \to \mathbb C$ which is conformal, one has
\begin{equation}\label{e:koebe}K^{-1} \leq \frac{|\psi'(z_1)|}{|\psi'(z_2)|} \leq K. \end{equation}
\end{proposition}
\begin{proof}
A classical corollary to the Koebe distorition theorem \cite{pommerenke-book}*{Cor 1.1.5} is that \eqref{e:koebe} holds for any $z_1,z_2 \in \mathbb D$ and $K=e^{6 \lambda(z_1,z_2)}$ where $\lambda$ is the hyperbolic metric in $\mathbb D$. However, there exists an absolute constant $r>0$ and points $z_I \in \mathbb D$ such that $T_I \subset \{ z \in \mathbb D : \lambda(z,z_I) \le r \}$, whence \eqref{e:koebe} follows with $K = e^{12r}$.
\end{proof}

\subsubsection{Example}\label{ss:example}
Next, to show   that   Theorem \ref{thm:weighted-est} is a non-trivial qualitative extension of Theorem \ref{thm:strong-quant}, let us construct a conformal map $\psi$ such that $|\psi'|^2$ belongs to $\B_{p}(\mathbb D)$ for all $p>1$, but not to $\B_1(\mathbb D)$. Define
	\[ \psi_0(z) = \frac{z}{\log z}.\]
This a conformal map on the disk $D(\frac 14, \frac 14)=\{z \in \mathbb C : |z-\frac 14| < \frac 14\}$ whose image is drawn in Figure \ref{f:domain}.
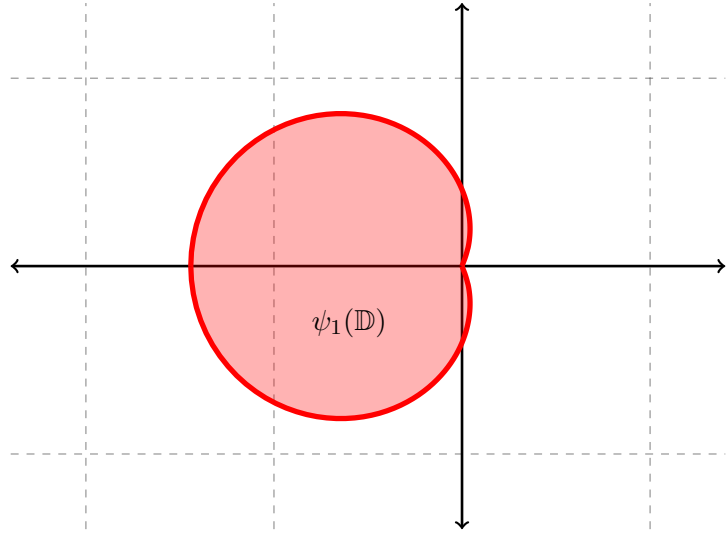
\begin{figure}\label{f:domain}
\begin{tikzpicture}[scale=5]
\draw[very thin,color=gray,step=.5cm,dashed] (-1.2,-0.7) grid (0.7,0.7);
\draw[<->, line width=1pt] (-1.2,0) -- (0.7,0);
\draw[<->, line width=1pt] (0,-0.7) -- (0,0.7);
\draw [red, line width=2pt,  domain=0:3.14, samples=75] 
 plot ({
	(
	(1+cos(\x r))*ln( cos(\x*0.5 r)*0.5 ) 
	+ sin(\x r)*\x*0.5
	)/
	(4*(ln( cos(\x*0.5 r)*0.5 ))^2+\x^2)
	} ,
	{
	(
	(sin(\x r))*ln( cos(\x*0.5 r)*0.5 ) 
	- (1+ cos(\x r))*\x*0.5
	)/
	(4*(ln( cos(\x*0.5 r)*0.5 ))^2+\x^2)
	});
\draw [red, line width=2pt,  domain=0:3.14, samples=75] 
 plot ({
	(
	(1+cos(\x r))*ln( cos(\x*0.5 r)*0.5 ) 
	+ sin(\x r)*\x*0.5
	)/
	(4*(ln( cos(\x*0.5 r)*0.5 ))^2+\x^2)
	} ,
	{
	(
	-(sin(\x r))*ln( cos(\x*0.5 r)*0.5 ) 
	+ (1+ cos(\x r))*\x*0.5
	)/
	(4*(ln( cos(\x*0.5 r)*0.5 ))^2+\x^2)
	});
\fill [red, fill opacity=0.3,  domain=0:3.14, samples=75] 
 plot ({
	(
	(1+cos(\x r))*ln( cos(\x*0.5 r)*0.5 ) 
	+ sin(\x r)*\x*0.5
	)/
	(4*(ln( cos(\x*0.5 r)*0.5 ))^2+\x^2)
	} ,
	{
	(
	(sin(\x r))*ln( cos(\x*0.5 r)*0.5 ) 
	- (1+ cos(\x r))*\x*0.5
	)/
	(4*(ln( cos(\x*0.5 r)*0.5 ))^2+\x^2)
	});
\fill [red, fill opacity=0.3,  domain=0:3.14, samples=75] 
 plot ({
	(
	(1+cos(\x r))*ln( cos(\x*0.5 r)*0.5 ) 
	+ sin(\x r)*\x*0.5
	)/
	(4*(ln( cos(\x*0.5 r)*0.5 ))^2+\x^2)
	} ,
	{
	(
	-(sin(\x r))*ln( cos(\x*0.5 r)*0.5 ) 
	+ (1+ cos(\x r))*\x*0.5
	)/
	(4*(ln( cos(\x*0.5 r)*0.5 ))^2+\x^2)
	});
\draw (-0.3,-0.15) node {$\psi_1(\mathbb D)$};
\end{tikzpicture}
\caption{$\psi_1(\mathbb D)$}
\end{figure}
Compute
	\[ |\psi_0'(z)|^2 = \left|\frac{\log z -1}{(\log z)^2} \right|^2 \sim \frac{1}{|\log z|^2}.\]
But $|\log(z)|^\alpha$ is integrable over $D(\frac 14,\frac 14)$ for every $\alpha \in \mathbb R$. Define $\psi_1(z)=\psi_0(\frac{z+1}{4})$ so that $\psi_1(\mathbb D)=\psi_0(D(\frac 14,\frac 14))$. Then $\psi_1$ is a conformal map on $\mathbb D$ with $|\psi_1'|^2 \in \B_{p}(\mathbb D)$ for each $p>1$. However, 
$\psi_1'(-1) = 0$ so $|\psi_1'|^2$ does not belong to $\B_1(\mathbb D)$.

\subsubsection{Asymptotically   conformal   domains}\label{ss:asymp-conf}
Given a domain $\Omega \subset \mathbb C$  which is bounded by a Jordan curve,   and two points $z_1,z_2 \in \partial \Omega$, define $\partial \Omega(z_1,z_2)$ to be the smaller arc of $\partial \Omega$ between the points $z_1$ and $z_2$. $\Omega$ is said to be asymptotically conformal if
 	\[ \max_{z \in \partial \Omega(z_1,z_2)} \frac{ |z-z_1| + |z-z_2| }{|z_1-z_2|} \to 1 \quad \mbox{as} \quad |z_1-z_2| \to 0.\]

\begin{lemma}\label{lemma:asymp-conf}
If $\Omega$ is asymptotically conformal, then $\abs{\psi'}^2 \in \B_{p}(\mathbb D)$ for every $p>1$.
\end{lemma}
\begin{proof}In this case, $\log \psi'$ belongs to the little Bloch space \cite{pommerenke-VMO}. But it is well-known \cite[Theorem 8.15]{zhu2007} that this is equivalent to $\log \psi'$ belonging to $\VMOA$ defined with respect to the Bergman metric. More precisely, $\log \psi' \in \text{Hol}(\mathbb{D})$ and satisfies for any fixed $r>0$:
\[    \lim_{R \rightarrow 0^{-}}\sup_{\substack{z:\\1>|z| \geq R}} \frac{1}{|\beta(z,r)|}                         \int_{\beta(z,r)} |\log \psi' -\log \psi'_z|^2 \, dA=0. \]
Here $\beta(z,r)$ denotes a disk in the hyperbolic metric centered at $z$ and of radius $r$, while for any function $f$, we write $f_z$ to indicate $|\beta(z,r)|^{-1} \int_{\beta(z,r)} f \, dA$, omitting the implicit dependence on $r$. We say $f \in \BMO$ if $f$ satisfies the weaker condition, that
\[ \norm{f}_{\BMO} := \sup_{z \in \mathbb{D}}\frac{1}{|\beta(z,r)|}                         \int_{\beta(z,r)} |f -f_z|^2 \, dA \]
is finite. Different choices of $r$ give rise to equivalent norms, so this space (along with $\VMOA$) is independent of $r>0$.
Moreover, any function in $\VMOA$ can be arbitrarily well-approximated in the $\BMO$ norm by elements of $C(\overline{\mathbb{D}})$ \cite[Proposition 8.12]{zhu2007}, and hence in particular by polynomials, which belong to $\VMOA$. An additional key fact \cite[pp.218]{zhu2007} that we will use is that any $f \in \BMOA= \BMO \cap \text{ Hol}(\mathbb{D})$ satisfies the Lipschitz-like condition with respect to the hyperbolic metric $\lambda$:
$$ |f(z_1)-f(z_2)| \leq C \|f\|_{\BMO} \lambda(z_1,z_2), \quad z_1,z_2\in \mathbb{D},$$
where $C$ is an independent constant. 

Let $p>1$ and write $\log \psi'=g+h$, where $g$ is a polynomial and $h \in \BMOA$ satisfying $\|h\|_{\BMO} \leq \frac{\log (3/2)}{6 C r} \min\{1, p-1\}.$ We will show that $|\psi'|^{2} \in \B_p(\mathbb D).$ Fix a Carleson tent $Q_I$ corresponding to $I \subset \mathbb{T}$ and let $I_j^k$ denote the dyadic descendants of $I$ belonging to the $k$-th generation, which  (so $k \geq 1$), $1 \leq j \leq 2^k.$ We then may choose $r>0$ and points $z_{j,k} \in T_{I_j^k}$ so 
 \[  T_{I_j^k} \subset \beta(z_{j,k},r),\quad |T_{I_j^k}| \sim |\
 \beta(z_{j,k},r)| \sim 2^{-2k} |Q_I| .\] 
Let $z_I$ be such a point corresponding to $T_I=T_{I_1^{0}}$. It is straightforward to see then that if $z \in T_{I_j^k},$ there holds 
\[\lambda(z, z_I) \leq (2k+1)r,  \quad |h(z)-h(z_I)| \leq 3 C r k \|h\|_{\BMO}.        \]
Then, we can directly estimate:
\begin{align*}
\frac{1}{|Q_I|} \int_{Q_I} |\psi'|^{2} & \leq e^{2\|g\|_{\infty}} \cdot e^{2 \text{Re}\, h(z_I)} \cdot\frac{1}{|Q_I|} \int_{Q_I} e^{2\text{Re}\,(h-h(z_I))}\\
& \leq \frac{e^{2(\|g\|_{\infty}+ \text{Re} \, h(z_I))}}{|Q_I|} \sum_{k\geq0} \sum_{j=1}^{2^k} \int_{T_{I_j^k}}e^{2|h-h(z_I)|} \\
& \lesssim e^{2(\|g\|_{\infty}+\text{Re} \, h(z_I))} \sum_{k\geq0} 2^{-k} e^{6 Crk \|h\|_{\BMO}}\\
& = e^{2(\|g\|_{\infty}+ \text{Re} \, h(z_I))} \sum_{k\geq0} \left( \frac{e^{6 Cr \|h\|_{\BMO}}}{2}\right)^k
\\ & \lesssim e^{2(\|g\|_{\infty}+ \text{Re} \, h(z_I))}
\end{align*}
A similar computation shows that

\[   \left(\frac{1}{|Q_I|} \int_{Q_I} |\psi'|^{-2/(p-1)}                    \right)^{p-1} \lesssim e^{2(\|g\|_{\infty}-\text{Re}\, h(z_I))} ,\]
completing the proof. 

\end{proof}
  We remark that a very similar result appears in the literature in \cite[Lemma 2.4]{limani2023}

\subsection{Uniform Domains and the $\D_p(\Omega)$ class}
We begin with the definition of uniform domains, which have also been called $(\varepsilon,\infty)$ domains in the literature.   Uniform domains, which coincide with quasicircles in the plane (see Proposition \ref{p:uniform} below), were first introduced by F. John in \cite{FJohn}, but have since found applications to many different areas of function theory (see, for example \cite{jones80, jones81}). A simply connected domain $\Omega$ bounded by a Jordan curve is called a \emph{uniform domain} if there exists $C>0$ such that for each $z_1,z_2 \in \partial\Omega$,
		\[ \sup_{z \in \partial \Omega(z_1,z_2)} |z_1-z|+|z_2-z|\le C |z_1-z_2|. \]
  Recall the notation $\partial \Omega(z_1,z_2)$ from Subsection \ref{ss:asymp-conf}. 

The next proposition, which is a well-known result originally due to Ahlfors, gives an equivalent definition for uniform domains.
\begin{proposition}\label{p:uniform}\cite{pommerenke-book}*{p. 94}
Let $\Omega$ be a bounded simply connected Jordan domain and $\psi$ a conformal map from $\mathbb D$ onto $\Omega$. The following are equivalent.
\begin{itemize}
	\item[A.]$\Omega$ is a uniform domain.
	\item[B.] There exists $\Psi : \mathbb C \to \mathbb C$ quasiconformal such that $\Psi = \psi$ on $\mathbb D$.
\end{itemize}
\end{proposition}

As discussed in the introduction, while the definition for $\D_p(\Omega)$ in \eqref{e:Dp} is more intrinsic and geometric, the $\B_p(\Omega)$ characteristic defined in \eqref{e:BpO} has many advantages over it.   Not only does it allow us to leverage the dyadic structure of $Q_I$ on $\mathbb D$, which makes for cleaner computations, it also clearly connects to the regularity of the conformal map $\psi$ at the boundary, which is crucial for our results.  

To prove the equivalence of these classes for uniform domains, we will use part B. of Proposition \ref{p:uniform}. The two important properties of quasiconformal maps $\Psi$ which we will need are
\begin{enumerate}
\item $|J \Psi|:= |\partial \Psi|^2-|\overline{\partial} \Psi|^2$ belongs to $\A_\infty(\mathbb{C})$;
\item $\Psi$ is quasisymmetric.
\end{enumerate}
Let us explain these two properties. First, there are many equivalent definitions of the weight class $\A_\infty(\mathbb C)$ \cite{gc-book}*{Corollary IV.2.13 and Theorem IV.2.15}. We will say a weight $v$ belongs to $\A_\infty(\mathbb C)$ if there exists $C,\delta>0$ such that for each cube $Q \subset \mathbb C$ and each measurable set $E \subset Q$,
	\[ \frac{|E|}{|Q|} \le C \left( \frac{\int_E v \, dA}{\int_Q v \, dA} \right)^{\delta}.\]
Second, $f: \mathbb{C} \rightarrow \mathbb{C}$ is \emph{quasisymmetric} if there exists an increasing function $\eta: \mathbb{R}^{+} \rightarrow \mathbb{R}^{+}$ such that for $z_0,z_1,z_2 \in \mathbb{C}$,

$$ \frac{|f(z_1)-f(z_0)|}{|f(z_2)-f(z_0)|} \leq \eta \left(\frac{|z_1-z_0|}{|z_2-z_0|} \right).$$ 
It follows by taking $z_i=f^{-1}(w_i)$ that if $f$ is quasisymmetric, so is $f^{-1}$, though with different $\eta.$ 

(1) is a well-known result due to Gehring \cite{gehring}, see also \cite{astala-book}*{Theorem 13.4.2}. For (2), see \cite{astala-book}*{Theorem 3.5.3}.
With these two facts, we can prove the following lemma for $\psi$   which will   be crucial in showing the sharpness of our condition as well as the equivalence of $\D_p(\Omega)$ and $\B_p(\Omega)$.

\begin{lemma}\label{lem:ainfintytents}
Let $\Omega \subset \mathbb{C}$ be a bounded, simply connected uniform domain with $\psi: \mathbb{D} \rightarrow \Omega$ conformal and fix $I \subseteq \mathbb{T}.$ There exists $\delta>0$ such that for all measurable $E \subseteq Q_I$,
$$ \frac{|E|}{|Q_I|} \lesssim \left(\frac{|\psi(E)|}{|\psi(Q_I)|} \right)^\delta.$$

\begin{proof}
Let $I$ be an interval in $\mathbb T$. There exists a Euclidean cube $P \supseteq Q_I$ with $|P| \sim |Q_I|$. Since $|J \Psi| \in \A_{\infty}(\mathbb{C})$, there exists $\delta>0$ such that for any $E \subset Q_I \subset P$,
\begin{align*}
\frac{|E|}{|Q_I|} & \lesssim \frac{|E|}{|P|} \lesssim \left(\frac{\int_{E} |J \Psi| \, dA}{\int_{P} |J\Psi| \, dA} \right)^\delta \\
& \leq \left(\frac{\int_{E} |J \Psi| \, dA}{\int_{Q_I} |J\Psi| \, dA} \right)^\delta  = \left(\frac{\int_{E} |\psi'|^2 \, dA}{\int_{Q_I} |\psi'|^2} \, dA \right)^\delta\\
& = \left( \frac{|\psi(E)|}{|\psi(Q_I)|} \right)^\delta.
\end{align*}
\end{proof}
\end{lemma}

We are ready to prove the promised equivalence.
\begin{proposition}\label{prop:BpO-psi}
Suppose that $\Omega$ is a bounded, simply connected uniform domain and $\psi$ maps $\mathbb D$ onto $\Omega$ conformally. Then, for any weight $\sigma$,
	\[  \brk{\sigma}_{\B_p(\Omega)} \sim\brk{\sigma}_{\D_p(\Omega)}. \]
\begin{proof}
The upper inequality will follow from the fact that $\Psi$ is quasisymmetric and the lower inequality from a similar argument for $\Psi^{-1}$. To prove the upper bound, let $Q_I \subset \mathbb{D}$, and let $z_I$ be a point in $Q_I$ satisfying $|z_I-z|\sim \ell(I)$ for every $z \in \partial Q_I$. Let $r= \min_{z \in \partial Q_I}|\psi(z)-\psi(z_I)|$ and $R= \max_{z \in \partial Q_I}|\psi(z)-\psi(z_I)|$. The quasisymmetry of $\Psi$ shows that $r \sim R$ and clearly 

$$ D(\psi(z_I),r) \subset \psi(Q_I) \subset D(\psi(z_I),R).$$
Pick some $p \in \partial \Omega \cap D(\psi(z_I),R)$, and by the triangle inequality, $D(\psi(z_I),R) \subset D(p,2R)$. All in all,
$$ \psi(Q_I) \subset  D(p,2R) \cap \Omega   $$ and

\begin{align*}
|\psi(Q_I)|&  \geq |D(\psi(z_I),r)| \sim r^2 \sim R^2 \sim |D(p,2R)|\\
& \geq |D(p,2R) \cap \Omega|.
\end{align*}
This computation shows that every $\psi(Q_I)$ is contained in a disk $D$ centered on $\partial \Omega$ of comparable area, so $[\sigma]_{\B_p(\Omega)} \lesssim [\sigma]_{\D_p(\Omega)}.$ 

On the other hand, let $p \in \partial \Omega$ and $s>0$. We want to find $Q_I$ such that 
	\begin{equation}\label{e:claim2} D(p,s) \cap \Omega \subset \psi(Q_I), \qquad \abs{ D(p,s) \cap \Omega} \gtrsim \abs{\psi(Q_I)},\end{equation}
from which $[\sigma]_{\D_p(\Omega)} \lesssim [\sigma]_{\B_p(\Omega)}$ immediately follows. 
As before, set $r= \min_{|z-p|=s}|\Psi^{-1}(z)-\Psi^{-1}(p)|$ and $R= \max_{|z-p|=s}|\Psi^{-1}(z)-\Psi^{-1}(p)|$. Again, since $\Psi$ is quasisymmetric, $r \sim R$ and with $q=\Psi^{-1}(p) \in \mathbb T$,
$$ D(q,r) \subset \Psi^{-1}(D(p,s)) \subset D(q,R).$$
Simple geometry shows that there exists $Q_I \supset D(q,R) \cap \mathbb{D}$ with $|Q_I| \sim R^2$. Clearly,
$$ D(p,s) \cap \Omega \subset \psi(Q_I).$$
To establish \eqref{e:claim2}, it remains to show $|\psi(Q_I)| \lesssim |D(p,s) \cap \Omega|.$ Thus, the proof is concluded by applying Lemma \ref{lem:ainfintytents} to obtain

$$1 \lesssim \left( \frac{|D(q,r) \cap \mathbb{D}|}{|Q_I|} \right)^{1/\delta} \lesssim\frac{|\psi(D(q,r) \cap \mathbb{D})|}{|\psi(Q_I)|} \lesssim \frac{|D(p,s) \cap \Omega|}{|\psi(Q_I)|}.  $$
\end{proof}
\end{proposition}

\section{Limited range weighted estimates for $\BP_{\Omega}$}\label{s:main-proof}
In addition to the weight classes $\B_p(\Omega)$ defined in \eqref{e:BpO}, for $q_0 \ge 1$, define
	\[ \B_{q_0^+}(\mathbb D) = \bigcap_{q > q_0} \B_q(\mathbb D),\]
and for $1<s<\infty$, define the reverse H\"older class $\RH_s(\Omega)$ to be the weights $\sigma$ on $\Omega$ such that the characteristic
	\[ \brk{\sigma}_{\RH_s(\Omega)} = \sup_{I \subseteq \mathbb{T}} \avg{ \sigma }_{s,\psi(Q_I)} \left( \avg{ \sigma }_{\psi(Q_I)} \right)^{-1} \]
is finite. The goal of this section is to prove the following of which Theorem \ref{thm:mainsufffull} from the introduction is the special case $p_0=q_0=1$.
\begin{theorem}\label{thm:weighted-est}
Let $\Omega$ be a bounded, simply connected domain and $\psi$ a conformal map from $\mathbb D$ onto $\Omega$. Let $1 \le p_0 < 2$ and set $q_0 = \frac{p_0}{2-p_0}$. If $\abs{\psi'}^2 \in \B_{q_0^+}(\mathbb D)$, $p_0 < p < p_0'$, and
	\begin{equation}\label{e:strong-qual}\sigma \in \left\{ \begin{array}{ll} \B_p(\Omega) & p_0=1, \\[2mm] \B_{\frac{p}{p_0}}(\Omega) \cap \RH_{\frac{p_0'}{p_0'-p}}(\Omega), & p_0 > 1, \end{array} \right. \end{equation}
then $\BP_{\Omega} : L^p(\Omega,\sigma) \to L^p(\Omega,\sigma)$.\end{theorem}

If $\abs{\psi'}^2$ belonged to the smaller class $\B_{q_0}(\mathbb D)$, this theorem would be more or less trivial for the same reasons leading to Theorem \ref{thm:strong-quant} above. We outline that argument below, after the statement of Theorem \ref{thm:main}, the main technical result from which Theorem \ref{thm:weighted-est} follows. 

For two weights $u,v$ on $\mathbb D$ and a dyadic grid $\mathcal D$ define the weighted characteristic
	\[ \brk{u}_{\B_p(\mathcal D,v)} = \sup_{I \in \mathcal D} \avg{ u}_{v,Q_I} \avg{u^{-1}}_{v,\frac{1}{p-1},Q_I}.\]
We say $u$ belongs to $\B_p(\mathcal D,v)$ if the above characteristic is finite, and we drop the dependence on $v$ when $v \equiv 1$. Similarly, we let 
\[ \brk{u}_{\RH_s(\mathcal D,v)} = \sup_{I \in \mathcal D} \avg{ u}_{s,v,Q_I} \left(\avg{u}_{v,Q_I}\right)^{-1}.\]

\begin{theorem}\label{thm:main}
Let $\psi:\mathbb D \to \mathbb C$ be conformal, $\mathcal D$ a dyadic grid on $\mathbb T$, and $q_0 \ge 1$. Suppose that $v=\abs{\psi'}^2 \in \B_{q_0^+}(\mathbb D)$. If
\begin{equation}\label{e:main-ss} u \in \left\{ \begin{array}{ll} \B_2(\mathcal D,v),& q_0=1; \\[2mm] \B_{\frac{q_0+1}{q_0}}(\mathcal D,v) \cap \RH_{q_0}(\mathcal D,v) & 1<q_0<\infty, \end{array} \right.\end{equation}
then $u \in \B_2(\mathcal D,v)$.
\end{theorem}

If $v$ belonged to the smaller class $\B_{q_0}(\mathbb D)$, this result would be trivial. Indeed, for $q_0=1$, consult \eqref{e:introB1Bp}. For $q_0>1$, $r=\frac{1+q_0}{q_0}$ satisfies $-q_0=\frac{1}{1-r}$. Therefore,
	\begin{equation}\label{e:trivial} \avg{ u }_{Q_I} \avg{ u^{-1} }_{Q_I} \le \brk{v}_{\B_{q_0}}^{\frac{2}{q_0}} \avg{ u }_{q_0,v,Q_I} \avg{ u^{-1} }_{q_0,v,Q_I}\le \brk{v}_{\B_{q_0}}^{\frac{2}{q_0}} \brk{u}_{\RH_{q_0}(v)} \brk{u}_{\B_r(v)}.\end{equation}

The key idea in the proof of Theorem \ref{thm:main} is to use dyadic versions of weight characteristics and apply a dyadic regularization. We remark a similar idea appears in \cite{huowick2022} when proving weighted inequalities for the vector-valued Bergman projection with matrix weights. 

Recall the weighted dyadic maximal operators $M_{v}^{\mathcal{D}}$ from \eqref{e:max} and the following paragraph. Define the class $\B_{\infty}(\mathcal{D},v)$ to be those weights $u$ satisfying
\begin{equation}\label{e:Binfty} \brk{u}_{\B_\infty(\mathcal{D},v)}:= \sup_{I \in \mathcal{D}} \frac{\avg{ M_v^{\mathcal{D}}(u \chi_{Q_I}) }_{v,Q_I}}{\avg{u}_{v,Q_I}}  < \infty.\end{equation}

\begin{proposition}\label{prop:B2containment}
If there exists $C>0$ and $0<s<1$ such that
	\begin{equation}\label{e:weakRH} \avg{u}_{v,Q_I} \le C \avg{u}_{s,v,Q_I}, \qquad \forall I \in \mathcal D,\end{equation}
then $u \in \B_\infty(\mathcal D,v)$.
\end{proposition}
\begin{proof}
Let $I \in \mathcal D$ and $z \in Q_I$. Then,
	\[ M_v^{\mathcal D}(u\chi_{Q_I})(z) = \sup_{\substack{J \in \mathcal D: \\ J \subset I \\ z \in Q_J}} \avg{u}_{v,Q_J} \le C \left[ M_v^{\mathcal D}(u^s \chi_{Q_I})(z) \right]^{\frac{1}{s}}.\]
Multiplying the above display by $v(z)$, integrating over $z \in Q_I$, and using the fact that $M_v^{\mathcal D}$ is bounded on $L^{1/s}(\mathbb D,v)$ concludes the proof.
\end{proof}

An immediate consequence of Proposition \ref{prop:B2containment} is
	\begin{equation}\label{e:contain} \bigcup_{1 < p < \infty} \left\{ u : u \in \B_p(\mathcal D,v) \right\} \cup \left\{w^p : w \in \RH_p(\mathcal D,v) \right\} \subset \B_\infty(\mathcal D,v).\end{equation}
We say a weight $u$ satisfies the APR($\mathcal D$) condition if there exists $C>0$ such that for each $I \in \mathcal D$,
	\begin{equation}\label{e:APR} C^{-1} u(z_2) \le u(z_1) \le C u(z_2) \qquad \forall z_1,z_2 \in T_I.\end{equation}
This is a special case of the condition introduced in \cite{APR} with $\rho=\frac{1}{2}$ (though the condition in \cite{APR} is qualitatively independent of $\rho$ and ranges over all intervals).   Such weights  appear extensively in the literature and have been referred to as weights with bounded hyperbolic oscillation \cite{dayan2023, limani2023}.  

\begin{proposition}\label{prop:reverse Holder}
Suppose that $v$ is a doubling weight on $\mathcal{D}$, and $w \in \B_\infty(\mathcal D,v)$ satisfies the APR($\mathcal D$) condition. Then there exists $\tau>1$ only depending on $[w]_{\B_\infty(\mathcal D,v)}$ and $c_v$ so that $w \in \RH_{\tau}(\mathcal{D},v)$.

\begin{proof}
This is a straightforward adaptation of the argument in \cite{stockdale2023weighted}*{Lemma 2.8}. First note that, as a consequence of the APR($\mathcal D$) condition for $w$, for any Carleson box $Q_I$ with $I \in \mathcal D$ and $z \in Q_I$, we have $w(z) \lesssim M_{v}^{\mathcal D} (w \chi_{Q_I})(z).$ Indeed, note that if $z \in Q_I$, there is a unique interval $J \subset I$ in $\mathcal{D}$ so $z \in T_{J}.$ We then estimate, using \eqref{e:APR} and the second inequality in \eqref{e:doubling}:
\[ w(z) \lesssim  \avg{ w}_{T_{J},v} \lesssim \avg{w}_{Q_{J},v} \leq M^{\mathcal D}_{v}(w \chi_{Q_I})(z).\]
Therefore, we have reduced the problem to showing that for some $\tau>1$,
	\begin{equation}\label{e:reduce}\avg{M_{v}^{\mathcal D}(w \chi_{Q_I}) }_{\tau,v, Q_I} \lesssim \avg{w}_{v,Q_I}.\end{equation}

Fix $I \in \mathcal D$ and for each $\lambda>0,$ let $\Omega_\lambda= \{z \in Q_I: M_{v}^{\mathcal D}(w \chi_{Q_I})(z)> \lambda\}.$ For each $\lambda>0,$ write $\Omega_\lambda$ as a disjoint union of maximal $Q_{\lambda, j}$ and let ${\widehat Q}_{\lambda, j}$ denote the corresponding dyadic parents. Note that by maximality 
	\begin{align}\label{e:maximality1}&\avg{w \chi_{Q_I}}_{v, {\widehat Q}_{\lambda, j}} \leq \lambda, \\
	\label{e:maximality2}&M_v^{\mathcal D}(w \chi_{Q_I})(z) = M_v^{\mathcal D}(w \chi_{Q_{\lambda,j}})(z) \quad \mbox{for} \quad z \in Q_{\lambda,j}.\end{align}
Using the distribution function, for any $\tau>1$, split
\begin{equation}\label{e:distribution}\begin{aligned} \int_{Q_I} &[M_{v} (w \chi_{Q_I})]^{\tau} v \, dA \\
	&=  \int_{0}^{\avg{ w }_{v, Q_I}}  (\tau-1) \lambda^{\tau-2} \left( \int_{\Omega_\lambda} M_{v}^{\mathcal D} (w \chi_{Q_I}) v \, dA \right) \, d \lambda\\
	& + \int_{\avg{ w }_{v, Q_I}}^{\infty} (\tau-1) \lambda^{\tau-2} \left( \int_{\Omega_\lambda} M_{v}^{\mathcal D} (w \chi_{Q_I}) v \, dA \right) \, d \lambda.\end{aligned} \end{equation}
By virtue of the containment $\Omega_\lambda \subset Q_I$ and \eqref{e:Binfty}, the first integral is trivially bounded by $\brk{w}_{\B_\infty(\mathcal{D},v)}\avg{w}_{v,Q_I}^\tau v(Q_I)$. To handle the second integral, notice that for each $\lambda$ and $j$, by \eqref{e:maximality2}, \eqref{e:Binfty}, \eqref{e:maximality1}, and \eqref{e:doubling},
	\[ \int_{Q_{\lambda,j}} M_{v}^{\mathcal D} (w \chi_{Q_I}) v \, dA \le [w]_{\B_{\infty}(\mathcal D, v)} \int_{Q_{\lambda,j}} wv \, dA
	\le [w]_{\B_{\infty}(\mathcal D, v)} c_v \lambda \int_{Q_{\lambda,j}} v \, dA.\]
Summing this over $j$, we obtain
	\[ \begin{aligned} \int_{\avg{ w }_{v, Q_I}}^{\infty} &(\tau-1) \lambda^{\tau-2} \left( \int_{\Omega_\lambda} M_{v}^{\mathcal D} (w \chi_{Q_I}) v \, dA \right) \, d \lambda \\
		&\le [w]_{\B_\infty(\mathcal{D}, v)}\times c_v \times \frac{\tau-1}{\tau} \int_{\avg{ w }_{v, Q_I}}^{\infty} \tau \lambda^{\tau-1} v(\Omega_\lambda)\, d\lambda. \end{aligned} \]
Choosing $\tau = \frac{2c_{v} [w]_{\B_\infty(\mathcal{D},v)} }{2 c_{v}[w]_{\B_\infty(\mathcal{D},v)} -1}$, the final term is bounded by $\frac{1}{2}$ times the LHS of \eqref{e:distribution}; therefore \eqref{e:reduce} is established.
\end{proof} 
\end{proposition}

Given any two weights $u$, $v$, we introduce the following regularization associated to a dyadic grid $\mathcal{D}$:

\begin{equation}\label{e:reg} u_{\mathcal{D},v}:= \sum_{I \in \mathcal{D}}  \langle u \rangle_{v,T_I} \chi_{T_I}.\end{equation} 
When $v \equiv 1$, we simply write $u_{\mathcal{D}}.$ Let us list some elementary properties of the regularizations. 
\begin{lemma}\label{lemma:reg}
\begin{itemize}
\item[(i)] $u_{\mathcal D,v}$ satisfies the APR($\mathcal D$) condition.
	\item[(ii)] For each $I \in \mathcal D$, $\avg{u}_{v,Q_I} = \avg{u_{\mathcal D,v}}_{v,Q_I}$.
	\item[(iii)] For $0<s\le 1$, $(u^s)_{\mathcal D,v} \le u_{\mathcal D,v}^s$.
	\item[(iv)] If $u$ satisfies \eqref{e:weakRH}, then so does $u_{\mathcal D,v}$. 
\end{itemize}
\end{lemma}
\begin{proof}
Statements (i) and (ii) are trivial. Statement (iii) follows from H\"older's inequality. Statement (iv) follows by applying (ii), \eqref{e:weakRH}, (ii) again, and (iii) to obtain
	\[ \avg{u_{\mathcal D,v} }_{v,Q_I} = \avg{u}_{v,Q_I} \le C \avg{u}_{s,v,Q_I} = C \avg{ (u^s)_{\mathcal D,v} }_{v,Q_I}^{\frac{1}{s}} \le C \avg{u_{\mathcal D,v} }_{s,v,Q_I}.\]
\end{proof}
We are now ready for the final ingredient to the proof of Theorem \ref{thm:main}.

\begin{proposition}\label{prop:reg}
Let $v = \abs{\psi'}^2$ for $\psi$ which is conformal on $\mathbb D$. Let $1 \le q < \infty$ such that $u^q$ satisfies \eqref{e:weakRH}. Then there exists $r>q$ such that
	\[ \avg{ u }_{Q_I} \lesssim [v]_{\B_{r}(\mathbb D)}^{\frac{1}{r}} \avg{ u }_{q,v,Q_I}, \qquad \forall I \in \mathcal D.\]
\end{proposition}
\begin{proof}
Let $I \in \mathcal D$ and notice for any $r>1$,
\begin{equation}\label{e:nothing}\begin{aligned} \avg{ u }_{Q_I} &= \avg{ u_{\mathcal D} }_{Q_I} \\
			&= \avg{ u_{\mathcal D}v^{\frac{1} {r}} v^{ -\frac{1}{r}} }_{Q_I} \\
			&\le \avg{  u_{\mathcal D}^{r} v }_{Q_I}^{\frac 1r} \avg{ v^{-\frac{1}{r-1}} }_{Q_I}^{\frac {r-1}{r}} \\
			&\le \brk{v}_{\B_r}^{\frac 1r} \avg{ u_{\mathcal D} }_{r,v,Q_I}.
			\end{aligned} \end{equation}
However, since $v= \abs{\psi'}^2$ and $\psi$ is conformal, by the Koebe distortion theorem (Proposition \ref{KoebeDistort}), $\avg{ u }_{T_I} \sim \avg{ u }_{v,T_I}$. Thus $u_{\mathcal D} \sim u_{\mathcal D,v}$ which combined with \eqref{e:nothing} implies
	\[ \avg{ u }_{Q_I} \lesssim \brk{ v }_{\B_r}^{\frac 1 r} \avg{ u_{\mathcal D,v} }_{r,v, Q_I}.\]
By Lemma \ref{lemma:reg} (i) and (iv), $(u^q)_{\mathcal D,v}$ satisfies the APR($\mathcal D$) condition and \eqref{e:weakRH}, so Propositions \ref{prop:B2containment} and \ref{prop:reverse Holder} provide $\tau >1$ such that 
	\[ \avg{u_{\mathcal D,v}}_{q\tau,v,Q_I}^q \le \avg{(u^q)_{\mathcal D,v}}_{\tau,v,Q_I} \lesssim \avg{(u^q)_{\mathcal D,v}}_{v,Q_I} = \avg{u}_{q,v,Q_I}^q.\]
Taking $r=q\tau>q$ concludes the proof.
\end{proof}

Now the proof of Theorem \ref{thm:main} is immediate. Let $q_0 \ge 1$ and let $u$ belong to the RHS of \eqref{e:main-ss}. By \eqref{e:contain}, $u^{q_0}$ and $u^{-q_0}$ both satisfy  \eqref{e:weakRH}. Therefore, applying Proposition \ref{prop:reg} to $u$ and $u^{-1}$ with exponent $q=q_0$, gives $r_1,r_2> q_0$ such that \eqref{e:nothing} holds for the respective weights and parameters. Therefore, setting $r^* = \min\{r_1,r_2\} > q_0$, for each $I \in \mathcal D$,
	\begin{equation}\label{e:uup}
	\begin{aligned} \avg{ u }_{Q_I} \avg{ u^{-1} }_{Q_I}
		&\lesssim \brk{v}_{\B_{r^*}(\mathbb D)}^{\frac{2}{r^*}} \avg{u}_{q_0,v,Q_I} \avg{ u^{-1} }_{q_0,v,Q_I} \\
		& \le \brk{v}_{\B_{r^*}(\mathbb D)}^{\frac{2}{r^*}}  \left\{ \begin{array}{ll} \brk{u}_{\RH_{q_0}(\mathcal D,v)} \brk{u}_{\B_{\frac{q_0+1}{q_0}}(\mathcal D,v)}, & q_0>1; \\
				\brk{u}_{\B_2(\mathcal D,v)}, & q_0=1. \end{array} \right. \end{aligned} 
	\end{equation}

To derive Theorem \ref{thm:weighted-est} from Theorem \ref{thm:main}, let $\sigma$ be as in \eqref{e:strong-qual} with $p=2$. Setting $u = \sigma \circ \psi$ and $v = |\psi'|^2$, \eqref{e:intro-change} implies that $u$ belongs to the RHS of \eqref{e:main-ss} for any dyadic grid $\mathcal{D}$. Therefore, Theorem \ref{thm:main} places $u = \sigma \circ \psi \in \B_2(\mathcal{D},v)$, and taking $\mathcal{D}=\mathcal{D}_j$ for $j=1,2$ yields $u \in \B_2(\mathbb{D})$. By \eqref{e:w} $\BP_{\mathbb D}$ is bounded on $L^2(\mathbb D,u)$, and applying \eqref{e:transform} yields the conclusion when $p=2$. The case of general $p \in (p_0,p_0')$ follows by extrapolation \cite{cmp2011}*{Theorems 3.22 and 3.31} since $\{ \psi(Q_I)\}$ is a Muckenhoupt basis by Proposition \ref{prop:muck}.

Theorem \ref{thm:weighted-est} can also be proved without extrapolation. A more general statement than \eqref{e:main-ss} holds in Theorem \ref{thm:main}, namely that if
	\begin{equation}\label{e:main-ss2} v=\abs{\psi'}^2 \in \B_{q_0^+}(\mathbb D), \ \mbox{and} \ u \in \left\{ \begin{array}{ll} \B_{\frac{p}{p_0}}(\mathcal D,v) \cap \RH_{\frac{p_0'}{p_0'-p}}(\mathcal D,v) & p_0 > 1, \\ \B_p(\mathcal D,v) & p_0=1, \end{array} \right.\end{equation} 
then $uv^{1-p/2} \in \B_p(\mathcal D)$, where $p_0$ is related to $q_0$ by $\frac{p_0'}{2}=q_0'$. Assuming this, Theorem \ref{thm:weighted-est} can be proved in the same way as $p=2$ above. To prove that \eqref{e:main-ss2} implies $uv^{1-p/2} \in \B_p(\mathcal D)$, all that is needed is to tediously check many exponents and use the following generalization of Proposition \ref{prop:reg}.

\begin{proposition}\label{prop:reg2}
Let $v = \abs{\psi'}^2$ for $\psi$ which is conformal on $\mathbb D$ and $\theta < \frac 12$. Let $q \ge 1$ such that $(1-\theta)q' \ge 1$ and $u^q$ satisfies \eqref{e:weakRH}. Then there exists $r>\left( (1-\theta) q' \right)'$ such that
	\[ \avg{ uv^\theta }_{Q_I} \lesssim [v]_{\B_{r}(\mathbb D)}^{\frac{1}{r}} \avg{ u }_{q,v,Q_I} \avg{ v^{-1} }_{\frac{1}{r-1},Q_I}^{-\theta}, \qquad \forall I \in \mathcal D.\]
\end{proposition}
Proposition \ref{prop:reg} is the special case of $\theta=0$. Its proof is the same, using the regularization $u_{\mathcal D,v^\theta}$ rather than $u_{\mathcal D}$. Now assuming \eqref{e:main-ss2}, apply Proposition \ref{prop:reg2} to $(u,q_1,\theta_1)$ and $(u^{-\frac{1}{p-1}},q_2,\theta_2)$ with 
\[ \textstyle q_1 = \frac{p_0'}{p_0'-p}, \quad q_2 = \frac{p-1}{\frac{p}{p_0}-1}, \quad \theta_1 = 1-\frac{p}{2}, \quad \theta_2 = \frac{-1}{p-1} \left(1- \frac{p}{2} \right). \]
The key computation is that $((1-\theta_j) q_j')'=q_0$ for $j=1,2$, and then a similar argument to \eqref{e:uup} shows $uv^{1-p/2} \in \B_p(\mathcal D)$.

\section{Converse to Theorem \ref{thm:weighted-est} on uniform domains}\label{sec:necessity}
In this section, we will prove the following result showing that on uniform domains, Theorem \ref{thm:weighted-est} is sharp. Define the reverse H\"older characteristic with respect to disks centered on the boundary by
	\[ [\sigma]_{\mathrm{RHD}_s(\Omega)} = \sup_D \avg{\sigma}_{s,D \cap \Omega} \left( \avg{ \sigma}_{D \cap \Omega}\right)^{-1}.\]
\begin{theorem}\label{thm:TFAE}
Let $\Omega$ be a bounded, simply connected uniform domain and $\psi$ a conformal map from $\mathbb D$ onto $\Omega$. Suppose $1 \le p_0 < 2$ and set $q_0 = \frac{p_0}{2-p_0}$. The following are equivalent.
\begin{itemize}
	\item[A.] $\abs{\psi'}^2 \in \B_{q_0^+}(\mathbb D)$.
	\item[B.] For $p_0 < p < p_0'$, if
	\[\sigma \in \left\{ \begin{array}{ll} \B_p(\Omega) & p_0=1, \\[2mm] \B_{\frac{p}{p_0}}(\Omega) \cap \RH_{\frac{p_0'}{p_0'-p}}(\Omega), & p_0 > 1, \end{array} \right.\]
	then $ \BP_{\Omega} : L^p(\Omega,\sigma) \to L^p(\Omega,\sigma)$.
	\item[C.] For $p_0 < p < p_0'$, if
	\[ \sigma \in \left\{ \begin{array}{ll} \D_p(\Omega) & p_0=1, \\[2mm] \D_{\frac{p}{p_0}}(\Omega) \cap \mathrm{RHD}_{\frac{p_0'}{p_0'-p}}(\Omega), & p_0 > 1, \end{array} \right.\]
	then $\BP_{\Omega} : L^p(\Omega,\sigma) \to L^p(\Omega,\sigma)$.
\end{itemize}
\end{theorem}
B. and C. are equivalent by Proposition \ref{prop:BpO-psi} (the same argument there works for $\RH_s(\Omega)$ and $\mathrm{RHD}_s(\Omega)$). A. implies B. by Theorem \ref{thm:weighted-est} so it remains to show that B. implies A.
We will need the following two lemmata to establish this remaining implication. The first is a consequence of the Koebe Distortion Theorem (Proposition \ref{KoebeDistort}).
\begin{lemma}\label{lemma:neigh}
Let $I_1,I_2$ be neighboring intervals in $\mathbb T$ and $I_0 = I_1 \cup I_2$. Then
 	 \[ |\psi(Q_{I_0})| \lesssim |\psi(Q_{I_1})| + |\psi(Q_{I_2})|.\]
\end{lemma}

The second concerns the maximal function
	\[ M_{\psi} f(z) = \sup_{I \subset \mathbb T} \chi_{\psi(Q_I)} \avg{f}_{\psi(Q_I)}, \qquad f \in L^1(\Omega).\]
\begin{lemma}\label{lemma:B1O}
Let $f \in L^1(\Omega)$ and $q<1$. Then,
	\[ (M_\psi f)^q \in \B_1(\Omega).\]
\end{lemma}

The proofs of Lemmata \ref{lemma:B1O} and \ref{lemma:neigh} are postponed to Section \ref{ss:B1O}. For now, let us prove following proposition, which will quickly show that B. implies A. It is loosely based on the necessity argument from \cite{jones-hom} concerning the homeomorphisms preserving $\operatorname{BMO}(\mathbb R)$.

\begin{proposition}\label{prop:homM}
Suppose $1\le s<\infty$ has the property that every $w \in \B_1(\Omega) \cap \RH_s(\Omega)$ has the additional property that $w \circ \psi$ is weakly doubling. Then for $v=\abs{\psi'}^2$, any $g \in L^1(\mathbb D,v)$, and any $q<\frac 1s$,
	\begin{equation}\label{e:necM} \avg{ M_v(\chi_{Q_I}g)}_{q,Q_I} \lesssim \avg{g}_{v,Q_I}.\end{equation}
\end{proposition}

\begin{proof}
Let $sq<1$, and $g \in L^1(\mathbb D,v)$. Set $\phi=\psi^{-1}$ and $f= g \circ \phi$. Fix an interval $I$.
Partition $I = I_1 \cup I_2$ where $|\psi(Q_{I_1})| = |\psi(Q_{I_2})|$. By Lemma \ref{lemma:neigh} these quantities are comparable to $|\psi(Q_I)|$. Let $e^{i\theta_j}$ be the center of $I_j$, and
partition $Q_I = Q_1 \cup Q_2$ where 
	\[ Q_j =\left\{ re^{i\theta} \in Q_I : |\theta-\theta_j| \le \frac{\ell(I_j)}{2} \right\}.\]
The sublinearity of the maximal function and the fact that $q < 1$ implies that for some $j \in \{1,2\}$,
	\begin{equation}\label{e:half} 2\int_{\psi(Q_I)} M_\psi( \chi_{\psi(Q_j)}f)^{q} \abs{\phi'}^2 \ge \int_{\psi(Q_I)} M_\psi( \chi_{\psi(Q_I)}f)^{q} \abs{\phi'}^2. \end{equation}
Without loss of generality, assume \eqref{e:half} holds for $j=1$ and set $w=M_\psi( \chi_{\psi(Q_1)}f)^{q}$. Then, by Lemma \ref{lemma:B1O}, $w^s \in \B_1(\Omega)$ so $w \in \B_1(\Omega) \cap \RH_s(\Omega)$. Take $J$ to be a neighbor of $I_2$ and $I$ with $\ell(J)=\ell(I)$, and we claim that
	\begin{equation}\label{e:claim-w} w \lesssim \left( \avg{f}_{\psi(Q_I)} \right)^{q} \quad \mbox{ on } \psi(Q_{J}).\end{equation}
Indeed, for $z \in \psi(Q_J)$ and any $K$ such that $z \in \psi(Q_K)$ and $K$ has nonempty intersection with $I_1$, $I_2$ must be contained in $K$. Therefore, $|\psi(Q_K)| \ge |\psi(Q_{I_2})| \gtrsim |\psi(Q_I)|$, which establishes \eqref{e:claim-w}.

Now we pull back with $\psi$. Set $u = w \circ \psi$  and notice that
	\[ \avg{ u }_{Q_{J \cup I}} \gtrsim \frac{1}{ |Q_I|} \int_{Q_I} w \circ \psi \gtrsim \frac{1}{|Q_I|} \int_{\psi(Q_I)} M_\psi( \chi_{\psi(Q_I)}f)^{q} \abs{\phi'}^2,\]
where the last inequality follows by changing variables and \eqref{e:half}. On the other hand, since we assume that $u$ is weakly doubling, by \eqref{e:claim-w},
	\[  \avg{ u }_{Q_{J \cup I}} \lesssim  \avg{ u }_{Q_{J}} \lesssim \left( \avg{f}_{\psi(Q_I)} \right)^{q}.\]
Combining these two displays, we obtain 
	\[  \left( \frac{1}{|Q_I|} \int_{\psi(Q_I)} M_\psi( \chi_{\psi(Q_I)}f)^{q} \abs{\phi'}^2 \right)^{\frac{1}{q}} \lesssim \avg{f}_{\psi(Q_I)}.\]
Changing variables, we obtain \eqref{e:necM}.
\end{proof}

\begin{proof}[Proof of B. implies A. in Theorem \ref{thm:TFAE}]
Let $1 \le p_0 <2$, and let 
	\[  w \in \left\{ \begin{array}{ll} \B_1(\Omega) & p_0=1; \\[2mm] \B_{1}(\Omega) \cap \RH_{q_0}(\Omega), & p_0 > 1. \end{array} \right.\] 
Then, since $\B_1(\Omega) \subset \B_{\frac{2}{p_0}}(\Omega)$, part B. implies that $\BP_{\Omega}$ is bounded on $L^2(\Omega, w)$, or equivalently by \eqref{e:transform} and \eqref{e:w}, $w \circ \psi \in \B_2(\mathbb{D}).$ Since any weight in $\B_2(\mathbb{D})$ is weakly doubling, we have inequality \eqref{e:necM} for any $q< \frac{1}{q_0}$. Now, let $p>q_0$ and select $g=v^{-\frac{p}{p-1}}$ so that on the one hand
	\[ \avg{g}_{v,Q_I} = \frac{ \int_{Q_I} v^{-\frac{p}{p-1}+1} \, dA}{\int_{Q_I} v \, dA } =  \frac{ \int_{Q_I} v^{-\frac{1}{p-1}} \, dA}{\int_{Q_I} v \, dA }.\]
On the other hand, we claim to have the pointwise domination $$M_{v}(\chi_{Q_I} g)(\zeta) \gtrsim g(\zeta) , \qquad \zeta \in Q_I.$$ To see this, fix $\zeta \in Q_I$ and find $J \subseteq I$ so that $\zeta \in T_{J}.$ Then, using the doubling property of $v$ from Lemma \ref{lem:ainfintytents} together with Proposition \ref{KoebeDistort}, we have 

\begin{align*}
g(\zeta)  \sim \avg{g}_{v,T_{J}}
 \lesssim \langle g \rangle_{v, Q_{J}}
 \lesssim M_v(\chi_{Q_I} g)(\zeta).
\end{align*}
Therefore, applying \eqref{e:necM} with $q= \frac 1p < \frac 1{q_0}$, we obtain
\[  \frac{ \int_{Q_I} v^{-\frac{1}{p-1}} \, dA}{\int_{Q_I} v \, dA } \gtrsim \langle M_v(\chi_{Q_I} g) \rangle_{q, Q_I} \gtrsim \left( \frac{1}{|Q_I|} \int_{Q_I} v^{-\frac{1}{p-1}} \, dA \right)^{p}.\]
Simple algebra and taking a supremum over all intervals $I$ then implies $v \in \B_p(\mathbb{D})$, and since $p>q_0$ was arbitrary we conclude $v \in \B_{q_0^{+}}(\mathbb{D})$.
\end{proof}

\begin{remark}\label{rem:nec} Notice that the only place we used that $\Omega$ was a uniform domain was in the Proof of B. implies A. to bound $g$ pointwise by the maxmal function, which relied on $v$ being doubling. So even in the case of non-uniform domain, we obtain the necessary condition
	\[ \avg{ M_v(\chi_{Q_I}v^{-\frac{p}{p-1}}) }_{\frac 1p,Q_I} \avg{ v}_{Q_I} \lesssim \avg{ v^{-\frac{1}{p-1}}}_{Q_I}, \quad p > q_0,\]
which reduces to the $\B_{q_0^+}(\mathbb D)$ condition if $M_v( \chi_{Q_I}v^{-\frac{p}{p-1}}) \gtrsim v^{-\frac{p}{p-1}}$, which holds in particular if $v$ is doubling.
\end{remark}

\subsection{Proofs of Lemmata \ref{lemma:neigh} and \ref{lemma:B1O}}\label{ss:B1O}
\begin{proof}[Proof of Lemma \ref{lemma:neigh}]
We may assume $\ell(I_1) \ge \ell(I_2)$ and $I_1$ is the right neighbor of $I_2$. Let $\theta_j$ be the center of $I_j$. Then $Q_{I_0}$ can be partitioned into 4 regions: $Q_{I_1}$, $Q_{I_2}$,
	\[ B = \left\{ re^{i\theta} : 1-\ell(I_1) \le r < 1-\ell(I_2), |\theta - \theta_2| \le \frac{\ell(I_2)}{2} \right\},\]
and
	\[ T = \left\{ re^{i\theta} : 1-\ell(I_0) \le r < 1-\ell(I_1), |\theta - \theta_0| \le \frac{\ell(I_0)}{2} \right\}.\]
\begin{figure}\label{f:neigh}
\begin{tikzpicture}[scale=0.75]
\fill[gray,fill opacity=0.3] (0,0) rectangle (5,5);
\draw[gray,line width=2pt] (0,0) rectangle (5,5);
\draw (2.5,1.5) node {$Q_{I_1}$};
\fill[purple,fill opacity=0.3] (0,0) rectangle (-1,1);
\draw[purple,line width=2pt] (0,0) rectangle (-1,1);
\draw[purple] (0,0) rectangle (-1,1);
\draw (-0.5,0.5) node {$Q_{I_2}$};
\fill[blue,fill opacity=0.3] (0,1) rectangle (-1,5);
\draw[blue,line width=2pt] (0,1) rectangle (-1,5);
\draw (-0.5,2.25) node {$B$};
\fill[red,fill opacity=0.3] (-1,5) rectangle (5,6);
\draw[red,line width=2pt] (-1,5) rectangle (5,6);
\draw (2,5.5) node {$T$};
\fill[green,fill opacity=0.3] (-0.5,2.75) rectangle (5,5.5);
\draw[green,line width=2pt] (-0.5,2.75) rectangle (5,5.5);
\draw (2,4) node {$T_J$};
\end{tikzpicture}
\hspace{5ex}
\begin{tikzpicture}[scale=0.75]
\fill[gray,fill opacity=0.3] (0,0) rectangle (5,5);
\draw[gray,line width=2pt] (0,0) rectangle (5,5);
\draw (2.5,1.5) node {$Q_{I_1}$};
\fill[purple,fill opacity=0.3] (0,0) rectangle (-1,1);
\draw[purple,line width=2pt] (0,0) rectangle (-1,1);
\draw[purple] (0,0) rectangle (-1,1);
\draw (-0.5,0.5) node {$Q_{I_2}$};
\fill[blue,fill opacity=0.3] (0,1) rectangle (-1,5);
\draw[blue,line width=2pt] (0,1) rectangle (-1,5);
\draw (-0.5,3) node {$B$};
\fill[red,fill opacity=0.3] (-1,5) rectangle (5,6);
\draw[red,line width=2pt] (-1,5) rectangle (5,6);
\draw (2,5.5) node {$T$};
\fill[green,fill opacity=0.3] (-1,1) rectangle (1,2);
\draw[green,line width=2pt] (-1,1) rectangle (1,2);
\draw (0.5,1.5) node {$T_{J_1}$};
\fill[green,fill opacity=0.3] (-1,2) rectangle (3,4);
\draw[green,line width=2pt] (-1,2) rectangle (3,4);
\draw (1,3) node {$T_{J_2}$};
\fill[green,fill opacity=0.3] (-1,4) rectangle (7,8);
\draw[green,line width=2pt] (-1,4) rectangle (7,8);
\draw (3,7) node {$T_{J_3}$};
\end{tikzpicture}
\caption{$Q_{I_0} = Q_{I_1} \cup Q_{I_2} \cup B \cup T$ and $B \subset \cup_{k=1}^K T_{J_k}$.}
\end{figure}
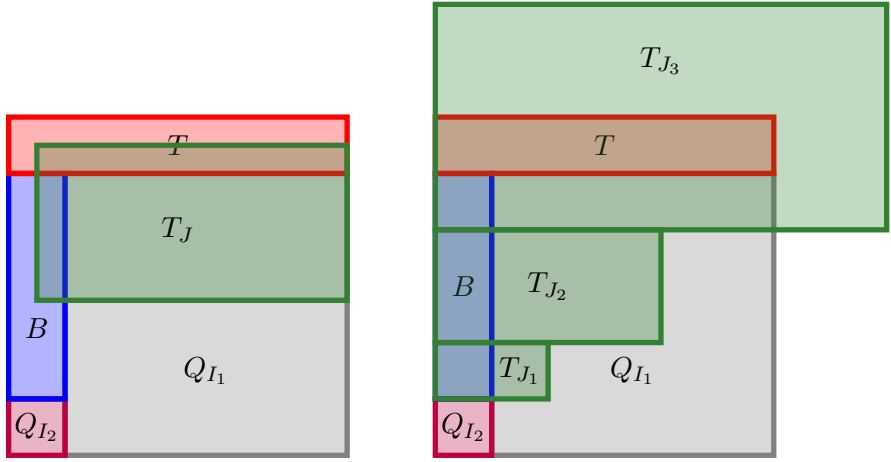
First, let $J$ be an interval with the same right endpoint as $I_1$ but with $\ell(J) = \frac{\ell(I_1) + \ell(I_0)}{2}$. There exists $z_1 \in T \cap T_J$ and $z_2 \in T_{I_1} \cap T_{J}$. Furthermore, since $T \subset T_{I_0}$, by Proposition \ref{KoebeDistort}, for any $z_3 \in T$ and $z_4 \in T_{I_1}$, 
	\[ \abs{\psi'(z_3)} \sim \abs{\psi'(z_1)} \sim \abs{\psi'(z_2)} \sim \abs{\psi'(z_4)}.\]
Since $|T| \lesssim |T_{I_1}|$ the above display implies $\abs{\psi(T)} \lesssim \abs{\psi(T_{I_1})} \le \abs{\psi(Q_{I_1})}$. 

To handle $B$, for each $k \in \mathbb N$, let $J_k$ be the interval with the same left endpoint as $I_2$, but with $\ell(J_k) = 2^k \ell(I_2)$. Then, $B \subset \cup_{k=1}^K T_{J_k}$ where $K = \lceil \log_2 \frac {\ell(I_1)}{\ell(I_2)} \rceil$. $\{T_{J_k}\}_{k=1}^K$ is pairwise disjoint and
	\[ |T_{J_k}| \lesssim |(Q_{I_1} \cup T) \cap T_{J_k}|.\]
Therefore, applying again Proposition \ref{KoebeDistort} and summing over $k$,
	\[ |\psi(B)| \le \sum_{k=1}^K |\psi(T_{J_k})| \lesssim \sum_{k=1}^K |\psi( (Q_{I_1} \cup T) \cap T_{J_k})| \lesssim \abs{\psi(Q_I)}.\]
\end{proof}
To prove Lemma \ref{lemma:B1O}, we will use
the following Vitali covering lemma for $\psi(Q_I)$.
\begin{lemma}\label{lemma:vitali} Let $\mathcal I$ be a finite collection intervals in $\mathbb T$. 
Then, there exist disjoint intervals $\{I_k\}_{k=1}^K \subset \mathcal I$, 
such that
	\[ \left\vert \bigcup_{I \in \mathcal I} \psi(Q_I) \right\vert \lesssim \sum_{k=1}^K \left\vert  \psi(Q_{I_k}) \right \vert,\]
where the implicit constant is absolute.
\end{lemma}
\begin{proof}
For any non-empty collection of intervals $\mathcal J$, let $I(\mathcal J)$ to be an interval in $\mathcal J$ with $|\psi(Q_{I(\mathcal J)})| = \max_{I \in \mathcal J} \abs{\psi(Q_I)}$. 
Let $\mathcal I_0 = \mathcal I$ and inductively define for $k \ge 0$,
	\[ I_k = I(\mathcal I_k), \quad \mathcal I_{k+1} = \{ I \in \mathcal I_k : I \cap I_k = \varnothing \}.\]
Since $\mathcal I$ is a finite collection, this process terminates after say $K$ steps once $\mathcal I_{K+1}=\varnothing$. 
By construction $\{I_k\}_{k=0}^K$ are disjoint. Fix $k$ and let $J_k^1$ and $J_k^2$ be the left and right neighbors of $I_k$ such that 
	\[  |\psi(Q_{J_k^1})| = |\psi(Q_{I_k})| = |\psi(Q_{J_k^2})|.\]  
Setting $J_k = J_k^1\cup I_k \cup J_k^2$, Lemma \ref{lemma:neigh} implies $\abs{\psi(Q_{J_k})} \lesssim \abs{\psi(Q_{I_k})}$. The lemma will be proved if we can show that for each $I \in \mathcal I$, there exists $k$ such that $I \subset J_k$.

To this end, given $I \in\mathcal I$, let $k$ be the first index such that $I \not\in \mathcal I_{k+1}$. Then $I \in \mathcal I_k$ and $I \cap I_{k} \ne \varnothing$. Suppose toward a contradiction that $I \backslash J_k \ne \varnothing$. Then $I \supsetneq J_k^{j}$ for some $j \in \{1,2\}$, but since $I \in \mathcal I_k$, 
	\[  \abs{\psi(Q_{I})} \le \abs{\psi(Q_{I_k})} = \abs{\psi(Q_{J_k^j})} < \abs{\psi(Q_{I})},\]
which is a contradiction.
\end{proof}
 As a consequence,
\begin{lemma} $M_\psi$ is of weak type $(1,1)$. \end{lemma}
\begin{proof} Let $f \in L^1(\Omega)$, $\lambda>0$, and $E=\{z \in \Omega : M_{\psi}f(z) > \lambda\}$. For each $z \in E$, there exists $I(z)$ such that $z \in \psi(Q_{I(z)})$ and $\avg{f}_{\psi(Q_{I(z)})} > \lambda$. Therefore $\{\psi(Q_{I(z)})\}_{z \in E}$ covers $E$. Now let $F$ be an arbitrary compact subset of $E$. $F$ can be covered by a finite subcollection of $\{\psi(Q_{I(z)})\}_{z \in E}$, say $\{\psi(Q_{I})\}_{I \in \mathcal I}$. Applying Lemma \ref{lemma:vitali} above, there exists a pairwise disjoint subcollection $\{I_k\}_{k=1}^K \subset \mathcal I$ such that
	\[ |F| \lesssim \sum_{k=1}^K |\psi(Q_{I_k})| \le \frac{1}{\lambda} \sum_{k=1}^K \int_{\psi(Q_{I_k})} |f| \, dA \le \frac{1}{\lambda} \norm {f}_{L^1(\Omega)}.\]
We conclude by taking the supremum over all $F$ which are compact subsets of $E$.
\end{proof}

\begin{proof}[Proof of Lemma \ref{lemma:B1O}]
The $\B_1(\Omega)$ condition can equivalently be stated as follows. There exists $C>0$ such that for each $z \in \Omega$ and $I$ such that $z \in \psi(Q_I)$,
	\[ \avg{ M_\psi f }_{q,\psi(Q_I)} \le C^{\frac{1}{q}} M_\psi f(z).\]
Fix such $z$ and $I$. Let $I_1,I_2$ be the right and left neighbors of $I$ such that $|\psi(Q_{I_j})| = |\psi(Q_I)|$. Setting $J = {I_1} \cup I \cup {I_2}$,
Lemma \ref{lemma:neigh} implies that
	\begin{equation}\label{e:DJ} |\psi(Q_I)| \sim |\psi(Q_J)|. \end{equation}  
Split $f=f_1+f_2$ where $f_1 = f \chi_{\psi(Q_J)}$.
Using the distribution function and the weak type estimate, for any $T>0$,
	\[ \begin{aligned} \int_{\psi(Q_I)} (M_\psi f_1)^q \, dA & \le \int_0^T + \int_T^\infty q \lambda^{q-1} | \{ z \in \psi(Q_I) : M_{\psi} f(z) > \lambda \} | \, d\lambda \\
	&\le T^q |\psi(Q_I)| + \norm{ f_1}_{L^1(\Omega)} \frac{q}{1-q} T^{q-1}. \end{aligned} \]
Taking $T= \norm{f_1}_{L^1(\Omega)} |\psi(Q_I)|^{-1}$, we obtain
	\[ \avg{ M_\psi f_1 }_{q,\psi(Q_I)} \lesssim \left( \frac{ \norm{ f_1}_{L^1(\Omega)} }{|\psi(Q_I)|} \right)^q \lesssim \avg{ f }_{\psi(Q_J)} \le M_\psi f(z),\]
where the second inequality comes from \eqref{e:DJ} and the support of $f_1$. For $f_2$, let $w \in \psi(Q_I) \cap \psi(Q_K)$. If $K \subset J$, then $\avg {f}_{\psi(Q_K)}=0$. Otherwise, $K$ exits $J$, but since it also has nonempty intersection with $I$, either $I_1$ or $I_2$ is contained in $K$. WLOG assume it is $I_1$. Then, setting $L = K \cup I$ and $K_1 = K \backslash I$, Lemma \ref{lemma:neigh}, applied to $K_1$ and $I$, implies
	\[ |\psi(Q_L)| \lesssim |\psi(Q_{K_1})| + |\psi(Q_I)| =  |\psi(Q_{K_1})| + |\psi(Q_{I_1})| \le 2|\psi(Q_{K})|.\]
Therefore, $\avg{f_2}_{\psi(Q_K)} \lesssim \avg{f}_{\psi(Q_L)}$ which implies $M_\psi f_2(w) \lesssim M_\psi f(z)$.
\end{proof}

\section{Proof of Theorem \ref{thm:weak}}\label{sec:weak}
Let $\sigma \in \B_1(\Omega)$. Set $u = \sigma \circ \psi$, $v=|\psi'|^2$ and $w=|\psi'|$. The weights $u$, $v$, and $w$ are defined on $\mathbb D$ so we omit the explicit reference to $\mathbb D$ in the $\B_p$ characteristics that follow.
We also identify weights with the absolutely continuous measure they induce via the notation $u(E) = \int_E u \, dA$ for measurable sets $E$.
The following weak-type estimate for the weighted dyadic maximal function can be proven using a standard maximal covering argument that is independent of the underlying measure $w \, dA$, which we omit. 
\begin{proposition}
Suppose $u, v, w$ are weights on $\mathbb{D}$ such that $uw \in \B_1(w)$ and $v=w^2$. Then, for all $f \in L^1(\mathbb{D}, uv),$

\begin{equation}
uv \left( \left \{ M_{w}^{\mathcal D} f > \lambda \right \} \right) \leq \frac{[uw]_{\B_1(w)}}{\lambda} \int_{\mathbb{D}} |f| \, uv \, dA.
\label{MaximalWeak}
\end{equation}
\end{proposition}

Next, we compute the analogue of \eqref{e:transform} for the weak-type estimate.
\begin{proposition}\label{prop:weak-transform} For each $\lambda>0$,
	\begin{equation}\label{e:weak-transform} \begin{aligned} &\sup_{ \norm {f}_{L^1(\Omega,\sigma)} =1} \sigma\left( \left\{ z \in \Omega : \abs{\BP_{\Omega} f(z)} > \lambda \right\} \right) \\
	= &\sup_{ \norm {g}_{L^1(\mathbb{D},uw)} =1} uv \left( \left\{ \zeta \in \mathbb D : \abs{ w(\zeta)^{-1} \BP_{\mathbb D} g(\zeta)} > \lambda \right\} \right) \end{aligned} \end{equation}
\end{proposition}
\begin{proof}
Let $\lambda>0$, $f \in L^1(\Omega,\sigma)$ with unit norm, and set 
	\[ E= \left\{ z \in \Omega:|\BP_{\Omega} f(z)|>\lambda \right\}.\] 
Let $g: \mathbb{D} \rightarrow \mathbb{C}$ be defined by $g=(f \circ \psi)\cdot \psi'$. Thus, by change of variable,
\[ 1=\int_{\Omega} \abs{f} \sigma \, dA= \int_{\mathbb{D}} \abs{g} u \abs{\psi'} \, dA, \quad 
 \sigma\left( E \right) = \int_{\psi^{-1}(E)} u\abs{\psi'}^2 \, dA.\]
Using the transformation law for the Bergman kernel (see for example \cite[p. 72]{lanzani-stein-2004}), one can conclude
\[\psi^{-1}(E) = \left\{\zeta \in \mathbb{D}: \, \frac{|\BP_{\mathbb{D}}g(\zeta)|}{ \abs{\psi'(\zeta)} } > \lambda \right\} .\]
\end{proof}

Equipped with Propositions \ref{prop:weak-transform} and \ref{MaximalWeak}, we are ready to prove Theorem \ref{thm:weak}.  First, we claim
	\begin{equation}\label{e:uwB1} \max\left\{ [u w]_{\B_1(w)},[u w]_{\B_1} \right\} \le [v]_{\B_1} [\sigma]_{\B_1(\Omega)}. \end{equation}
The estimate in \eqref{e:uwB1} for $[u w]_{\B_1}$ follows from \eqref{e:introB1Bp} with $p=1$ and a slight modification will show the estimate for $[u w]_{\B_1(w)}.$ 
Indeed, there holds for each $Q_I$,
\begin{align*}
[\sigma]_{\B_1(\Omega)} [v]_{\B_1}& \geq \frac{[v]_{\B_1}}{\int_{Q_I} v \, dA} \int_{Q_I} u v \, dA\cdot \|u^{-1} \|_{L^\infty(Q_I)}\\
& \geq \frac{1}{\inf_{Q_I} v} \avg{uv}_{Q_I}\|u^{-1} \|_{L^\infty(Q_I)}\\
& \geq \frac{1}{(\inf_{Q_I} v)^{1/2}} \avg{uv}_{Q_I}\|u^{-1} v^{-1/2} \|_{L^\infty(Q_I)}\\
& = \frac{\avg{ w }_{Q_I}}{\inf_{Q_I} w } \avg{uw}_{w,Q_I} \|u^{-1} w^{-1} \|_{L^\infty(Q_I)} \\
& \geq \avg{uw}_{w,Q_I}\|u^{-1} w^{-1} \|_{L^\infty(Q_I)}.
\end{align*}
Taking the supremum over $I$ establishes \eqref{e:uwB1}. Therefore, by \eqref{e:uwB1}, Proposition \ref{prop:weak-transform}, and the trivial estimates $c_w \le \brk{w}_{\B_1} \le \brk{v}_{\B_1}$, it is enough to show
\begin{equation}
uv \left( \left\{\zeta \in \mathbb{D}: \, \frac{|\BP_{\mathbb{D}}g(\zeta)|}{ w(\zeta)} > \lambda \right\} \right) \lesssim \frac{\left( c_w [w]_{\B_1} \right)^3 [uw]_{\B_1(w)}^{2}[uw]_{\B_1}}{\lambda} \int_{\mathbb{D}} |g| \, u w
\label{WeakTypePullback}
\end{equation}
Moreover, $\BP_{\mathbb{D}}$ is pointwise   bounded by   $\mathcal{A}_{\mathcal{D}_1}+  \mathcal{A}_{\mathcal{D}_2}$, defined by \eqref{e:A} with $v \equiv 1$
(see \cite[Lemma 5]{rtw2017}), so it is enough to prove \eqref{WeakTypePullback} for $\mathcal{A}_{\mathcal{D}}$ in place of $\BP_{\mathbb{D}}$ for a generic dyadic grid $\mathcal D$ and for $g$ positive. 

To this end, let $E_\lambda= \{ \zeta \in \mathbb D : M^{\mathcal{D}}_{w}(g w^{-1} )(\zeta)>\lambda\}$ and write $E_\lambda$ as a union of disjoint maximal Carleson boxes $Q_{\lambda,k}$ and let $\hat Q_{\lambda,k}$ denote the dyadic parents. Write
\[g=g_1+g_2, \quad g_1= g \chi_{E_\lambda^c}, \quad g_2= g \chi_{E_\lambda}.\]
Notice that if $I \in \mathcal{D}$, then either there exists $k$ so that $Q_I \subset Q_{\lambda,k}$ or $Q_I$ intersects $E_\lambda^c.$ 
In the latter case, note by definition of $Q_{\lambda,k}$ and the doubling of $w$, we have 
\begin{equation}    \avg{ g_1w^{-1} }_{w,T_I} 
\leq c_w \avg{ g_1 w^{-1} }_{w,Q_I} \le c_w \lambda, \label{GoodFunctionBound} \end{equation}
In the former case,
$ \avg{g_1 w^{-1} }_{w,T_I}=0$ since $Q_{\lambda,k} \subset E_\lambda$, so \eqref{GoodFunctionBound} holds for all $I \in \mathcal D$.
Similarly, for the function $g_2$, by maximality of $Q_{\lambda,k}$ and doubling of $w$,
\begin{equation} \avg{ g_2 w^{-1} }_{w,Q_{\lambda,k}}  
\leq c_w \avg{ g_2 w^{-1} }_{w, \hat Q_{\lambda,k}} 
< c_w \lambda.\label{GoodFunctionBound2} \end{equation}
Write
\begin{align*}
uv \left( \left\{\zeta \in \mathbb{D}: \, \frac{|\mathcal{A}_{\mathcal{D}}g(\zeta)|}{ w(\zeta)} > \lambda \right\} \right) & \leq uv \left( \left\{\zeta \in \mathbb{D}: \, \frac{|\mathcal{A}_{\mathcal{D}}g_1(\zeta)|}{ w(\zeta)} > \frac{\lambda}{2} \right\} \right) \\
&\quad + uv \left( \left\{\zeta \in E_\lambda^c: \, \frac{|\mathcal{A}_{\mathcal{D}} g_2(\zeta)|}{ w(\zeta)} > \frac{\lambda}{2} \right\} \right)\\
&\quad + uv(E_\lambda) \\
& := (I) + (II) +(III).
\end{align*}
$(III)$ is controlled by $\frac{ [uw]_{\B_1(w)}}{\lambda} \int_{\mathbb{D}} |g| \, u w$ using Proposition \ref{MaximalWeak}. 
To control $(I)$ and $(II)$, note that using the $\B_1$ condition for $w$,
\[\mathcal{A}_{\mathcal{D}}g_1(\zeta) w^{-1} (\zeta) \leq [w]_{\B_1} \sum_{I \in \mathcal{D}} \frac{\int_{Q_I}g_1 \, dA}{\int_{Q_I} w \, dA} \chi_{Q_I}(\zeta) = [w]_{\B_1} \mathcal{A}_{\mathcal{D}, w}(g_1 w^{-1} )(\zeta)\]
Recalling the dyadic regularization from \eqref{e:reg}, set $\widetilde {g_1} = (g_1)_{\mathcal D}$. By Lemma \ref{lemma:reg}.ii, $\int_{Q_I} \tilde g_1 \, dA = \int_{Q_I} g_1 \, dA$ for each $I \in \mathcal D$ whence
\[  \mathcal{A}_{\mathcal{D}, w}(g_1 w^{-1} ) = \mathcal{A}_{\mathcal{D}, w}(\widetilde{ g_1} w^{-1}).\]
Furthermore, for $I \in \mathcal D$ and $\zeta \in T_I \subset Q_I$, $w^{-1}(\zeta) \le [w]_{\B_1}\left( \int_{T_I} w \, dA \right)^{-1}$, which, combined with \eqref{GoodFunctionBound}, yields
\[ \widetilde{g_1} w^{-1} \le [w]_{\B_1}\sum_{I \in \mathcal{D}} \langle g_1 w^{-1} \rangle_{w,T_I} \chi_{T_I} \leq c_w \times [w]_{\B_1} \times \lambda.\] 
On the other hand, defining $ \widetilde{g_2}= \sum_k \langle g_2 \rangle_{Q_{\lambda,k}} \chi_{Q_{\lambda,k}} $, \eqref{GoodFunctionBound2} gives $\widetilde{g_2}w^{-1} \leq c_w[w]_{\B_1}\lambda$ and by the support condition on $g_2$, for $\zeta \in E_\lambda^c$ there holds 
\begin{align*}
\mathcal{A}_{\mathcal{D}}g_2(\zeta) w^{-1} (\zeta) & \leq [w]_{\B_1} \sum_{I \in \mathcal{D}} \frac{\int_{Q_I}g_2}{\int_{Q_I} w} \chi_{Q_I}(\zeta)\\
& \leq [w]_{\B_1}\sum_{k} \sum_{\substack{I\in \mathcal{D}\\ Q_I \supset Q_{\lambda_k}}}\frac{\int_{Q_I}g_2}{\int_{Q_I} w} \chi_{Q_I}(\zeta)\\
& \leq [w]_{\B_1} \mathcal{A}_{\mathcal{D}, w}(\widetilde{g_2} w^{-1} )(\zeta).
\end{align*}

Therefore,
	\[ \begin{aligned} (I) + (II) &\lesssim \sum_{j=1,2} \frac{\left([w]_{\B_1}\right)^2}{\lambda^2} \int_{\mathbb{D}} \left|\mathcal{A}_{\mathcal{D}, w}(\widetilde{g_j} w^{-1} )\right|^2 \, uv \, dA \\
		&\lesssim \sum_{j=1,2} \frac{\left(c_w [w]_{\B_1} [u w]_{\B_2(w)} \right)^2}{\lambda^2} \int_{\mathbb{D}} \left|\widetilde{g_j} w^{-1} \right|^2 \, uv \, dA \\
		&\lesssim \sum_{j=1,2} \frac{ \left(c_w [w]_{\B_1}\right)^3 \left( [u w]_{\B_2(w)} \right)^2}{\lambda} \int_{\mathbb{D}} \left|\widetilde{g_j} w^{-1}\right| \, uv \, dA, \end{aligned} \]
where the second inequality follows from \eqref{e:Aw} with $p=2$, $w \mapsto uw$, and $v \mapsto w$.
The proof of \eqref{WeakTypePullback} is concluded by again invoking the fact $uw \in \B_1$ to obtain 
	\[ \int_F \avg{h}_{F} uw \, dA \le \brk{uw}_{\B_1} \int_{F} |h|uw \, dA, \quad F \in \{Q_I,T_I\}_{I \in \mathcal D},\]
for any $h \in L^1(uw)$, whence
	\[ \int_{\mathbb{D}} \left|\widetilde{g_j} w^{-1}\right| \, uv \, dA \lesssim \brk{uw}_{\B_1} \int_{\mathbb{D}} \left\vert g \right\vert \, uw \, dA.\]

\end{document}